\documentclass[final]{siamltex}

\usepackage{graphicx}  

\usepackage{amssymb, amsmath}

\usepackage{color}

\usepackage[numbers, square, comma, sort&compress]{natbib}

\graphicspath{{./figures/}}

\usepackage{hyperref}  

\newcommand{\Par}[1]{\left(#1\right)}  

\newcommand{\pd}[2]{\frac{\partial #1}{\partial #2}}            
\newcommand{\pdn}[3]{{\frac{\partial^{#1}#2}{\partial#3^{#1}}}} 

\newcommand{\dt}{\Delta t}
\newcommand{\dx}{\Delta x}
\newcommand{\dy}{\Delta y}

\newcommand{\BigOh}{\mathcal{O}}

\newcommand{\R}{\mathbb{R}}             
\newcommand{\Z}{\mathbb{Z}}             

\newcommand{\h}{\bar{H}}                
\renewcommand{\u}{\bar{u}}              
\newcommand{\eps}{\varepsilon}          

\newcommand{\E}{\mathcal{E}}

\renewcommand\L{\mathcal{L}}            

\newtheorem{rmk}{Remark}
\newtheorem{thm}{Theorem}

\title{The Picard integral formulation of weighted essentially non-oscillatory schemes}

\author{David C. Seal\footnotemark[2]\ \footnotemark[4]
     \and Yaman G\"{u}\c{c}l\"{u}\footnotemark[2]\ \footnotemark[4]
     \and Andrew J. Christlieb \footnotemark[2]\ \footnotemark[3]\
     \footnotemark[4] }

\begin{document}

\maketitle

\begin{abstract}
High-order temporal discretizations for hyperbolic conservation laws have historically been formulated as either a method of lines (MOL) or a Lax-Wendroff method.  In the MOL viewpoint, the partial differential equation is treated as a large system of ordinary differential equations (ODEs), where an ODE tailored time-integrator is applied.  In contrast, Lax-Wendroff discretizations immediately convert Taylor series in time to discrete spatial derivatives.  In this work, we propose the Picard integral formulation (PIF), which is based on the method of modified fluxes, and is used to derive new Taylor and Runge-Kutta (RK) methods.  In particular, we construct a new class of conservative finite difference methods by applying WENO reconstructions to the so-called ``time-averaged'' fluxes.  Our schemes are automatically conservative under any modification of the fluxes,
which is attributed to the fact that classical WENO reconstructions conserve mass when coupled with forward Euler time steps. The proposed Lax-Wendroff discretization is constructed by taking Taylor series of the flux function as opposed to Taylor series of the conserved variables.  The RK discretization differs from classical MOL formulations because we apply WENO reconstructions to time-averaged fluxes rather than taking linear combinations of spatial derivatives of the flux.  In both cases, we only need one projection onto the characteristic variables per time step.  The PIF is generic, and lends itself to a multitude of options for further investigation.  At present, we present two canonical examples: one based on Taylor, and the other based on the classical RK method.  Stability analyses are presented for each method.  The proposed schemes are applied to hyperbolic conservation laws in one- and two-dimensions and the results are in good agreement with current state of the art methods.  

\end{abstract}

\renewcommand{\thefootnote}{\fnsymbol{footnote}}
\noindent\footnotemark[2]{
Department of Mathematics, Michigan State University, \\
\, \, East Lansing, Michigan, 48824, USA.} \\
\noindent\footnotemark[3]{
Department of Electrical and Computer Engineering, Michigan State University, \\
\, \, East Lansing, Michigan, 48824, USA.} \\
\noindent\footnotemark[4]{ This work was supported in part by AFOSR grants
FA9550-11-1-0281, FA9550-12-1-0343, and FA9550-12-1-0455, NSF grant
DMS-1115709, and MSU
Foundation grant SPG-RG100059.
}
\renewcommand{\thefootnote}{\arabic{footnote}}


\section{Introduction}
\label{sec:Introduction}

We begin our discussion with a 1D system of conservation laws defined by
\begin{equation}
\label{eqn:1D_system.PDE}
  q_{,t} + f(q)_{\!,\,x} = 0,
\end{equation}
where $q(t,x):\mathbb{R}^+\times \mathbb{R} \to \mathbb{R}^m$ is the unknown 
vector of $m$ conserved quantities and $f:\mathbb{R}^m \to \mathbb{R}^m$ is 
a prescribed flux function.  The starting point for this body of work is to 
perform formal integration of \eqref{eqn:1D_system.PDE} over a single interval 
$t \in [t^n, t^{n+1}]$, which defines the 
\emph{Picard integral formulation} (PIF)
of the 1D conservation law as
\begin{subequations}
\label{eqn:picard1d}
\begin{equation}
\label{eqn:picard1d-a}
    q(t^{n+1},x)\ =\ q(t^n,x)\ - \dt F^n(x)_{,\, x},
\end{equation}
where the \emph{time-averaged flux} is defined as
\begin{equation}
\label{eqn:picard1d-b}
    F^n(x) := \frac{1}{\dt} \int_{t^n}^{t^{n+1}} f( q(t, x ) )\, dt.
\end{equation}
\end{subequations}
We use the term Picard to emphasize that we 
propose solving, at least in spirit, an integral, as opposed to a differential equation in time.  
Our goal is to construct explicit time-stepping schemes that do not require an iterative
procedure.
The purpose of defining \eqref{eqn:picard1d} is that {\bf a new class of conservative finite difference methods} 
can be constructed by following the two step process:
\begin{enumerate}
    \item Approximate the time-averaged fluxes in \eqref{eqn:picard1d-b} with
        any temporal discretization.
        We offer the third-order Taylor and the classical fourth-order Runge-Kutta method as two canonical examples.
    \item Insert the result into the non-linear finite difference WENO 
	reconstruction procedure, which provides a conservative high-order approximation to 
	the spatial derivative in \eqref{eqn:picard1d-a}.  We present results that are fifth-order in space.
\end{enumerate}
The combination of these two steps will be called the 
Picard integral formulation of WENO (PIF-WENO), and we remark that 
the methods presented here can be extended to arbitrary order in space and time.
Some advantages that this formulation provide include:
\begin{itemize}
\item Methods that are constructed from this framework are automatically mass
conservative, independent upon how the time averaged fluxes are defined.
\item A variety of time
stepping methods including Runge-Kutta (RK), Taylor, multistep or even
multiderivative methods \cite{HaWa73} can be constructed from this framework.
Currently, we present two methods, one based on a Taylor formulation, and 
one based on an RK formulation.
\item The Picard integral formulation has the capacity to reduce computational 
complexity, including 
i)  smaller stencils, and 
ii) a reduction in the number of characteristic variable projections.
Both methods presented in this work only require a single projection onto
characteristic variables per time step.
\end{itemize}
%

The outline of this paper is as follows.  In \S\ref{sec:background} we present
background material and in \S\ref{sec:PIF1d}, we define two 1D formulations of
PIF-WENO, including a third-order Taylor and a fourth-order Runge-Kutta
discretization. 
We continue in \S\ref{sec:stability} with an analysis of the linear stability
properties of the proposed methods, and in \S\ref{sec:modified-flux-2d} we
extend the results to 2D.  Finally, numerical results and conclusions are
presented in \S\ref{sec:numerical-results} and \S\ref{sec:conclusions}.

\section{Background}
\label{sec:background}

Temporal discretizations for hyperbolic conservation laws are oftentimes
presented as falling into one of two categories: a method of lines (MOL) or a
Taylor/Lax-Wendroff method.  In a MOL formulation, one first discretizes in
space, and then applies an appropriate time integrator to the problem.
Popular high-order methods for hyperbolic conservation laws include finite
difference and finite volume weighted essentially non-oscillatory (WENO)
methods, discontinuous- and Petrov-Galerkin as well as piecewise parabolic
methods (PPMs).  The spatial discretizations defines a large system of ordinary differential
equations (ODEs), that is then integrated in time by a suitable ODE solver,
which is usually an explicit Runge-Kutta method.  MOL methods treat spatial
and temporal discretizations as separate operations, and therefore miss out on
coupling opportunities.  In this work, one of the two examples we present 
includes a new MOL approach that can be derived from the Picard integral
formulation.  Additionally, we derive Lax-Wendroff type methods from the
vantage the Picard integral formulation provides.

The basic procedure in a Lax-Wendroff (Taylor) method 
is to first discretize in time, and then immediately convert the
temporal derivatives to discrete spatial derivatives via the Cauchy-Kowalewski
procedure.  In doing so, space and time is coupled through the PDE, as opposed
to through relying on a system of ODEs.  The original
Lax-Wendroff method \cite{LxW60} is likely the first numerical method
based on this procedure, where second-order accuracy 
was obtained by including the evaluation of flux Jacobians.  
The coupling of arbitrarily high order Lax-Wendroff discretizations to finite volume 
methods was investigated in the 1980's with the original ENO methods of
Harten \cite{HaEnOsCh87}, and starting in the early 2000's, 
the so-called Arbitrary DERivativeâ (ADER) methods 
\cite{TitarevToro02-ader,ToroTitarev02,ToroTitarev05:JSP,TitarevToro05:JCP:systems,DuKa07,ToTi06,NoFi12,BaDiMeDuHuXu13}
have grown tremendously in popularity.
Recent work also includes coupling the Lax-Wendroff discretization to high-order
finite difference WENO 
\cite{QiuShu03,LiuChengShu09,LuQiu11,JiangShuZhang13} 
and discontinuous Galerkin 
\cite{QiuDumbserShu05,DuMu06,TaDuBaDiMu07,Qiu07:numcomp,GaDuHiMuCl11}
methods.  All high-order Lax-Wendroff discretizations rely on the evaluation
of Jacobians, Hessians and higher derivatives that inevitably show up for all
high-order single-step methods.  The finite difference WENO methods that fall
into this category typically compute the first derivative with a single
non-linear WENO reconstruction, and then use (centered) finite differences to
compute higher derivatives.  The finite volume flavored methods including the
discontinuous Galerkin and ADER methods \cite{DuBaDiToMuDi08,GaDuHiMuCl11}
rely on high-order solutions to generalized Riemann problems \cite{ToTi06}.
The common thread in these methods is that the time dependence is immediately
removed in the process of converting temporal to spatial derivatives.  When
compared with Runge-Kutta methods, there are no degrees of freedom to work
with because each coefficient in a Taylor series expansion is unique.

The purpose of this work is to define a new class of non-linear methods for 
hyperbolic conservation laws based on the Picard integral formulation, to
which many existing methods can be derived.  For example,
many finite volume formulations including the original Godunov scheme,
``high-order'' variations of the scheme \cite{VanLeer79,Ha84,Ro87,HaEnOsCh87,BeCoTr89,Me90},
Harten's original ENO method \cite{HaEnOsCh87}, and 
modern ADER methods \cite{TitarevToro05:JCP:systems,DuKa07,NoFi12,BaDiMeDuHuXu13}
work with time-averaged fluxes approximated through Taylor series.  In these
formulations, one takes the further step of integrating \eqref{eqn:picard1d-a}
over a control element and then applying the divergence theorem.  This defines
an evolution equation for the unknown cell averages, which is based on
integrating \eqref{eqn:picard1d-b} over the boundaries of each control
element.  The differences in the methods primarily come from the choice in
discretization of the time-averaged fluxes, as well as how each Riemann
problem is solved.  For example, in 1959 Godunov discretized the time-average
fluxes with constants; his insight was to leverage exact Riemann solvers for
piecewise constant functions.  In 1979 Van Leer \cite{VanLeer79} introduced
piecewise linear approximations to the solution, and therefore was able to
define a second-order version of Godunov's method.  A decade later, Harten,
Engquist, Osher and Chakravarthy \cite{HaEnOsCh87} defined an essentially
non-oscillatory method that used piecewise polynomial reconstructions of the
solution from cell averages, and therefore were able to extend Godunov's
method to arbitrary order.  In their seminal paper, time was discretized via
the Cauchy-Kowalewski/Lax-Wendroff procedure, which resulted in a
single-stage, single-step method \cite{HaEnOsCh87}.  A further decade later,
the now classical ADER formulation was defined \cite{TitarevToro02-ader},
where additional Arbitrary DERivative Riemann problems were introduced in
order to define an arbitrary order, single-step finite volume method.
%
Our focus is not on finite volume methods, nor their discontinuous Galerkin
cousins, but on conservative finite difference methods.

The vantage the Picard integral formulation provides allows one to 
produce a variety of conservative finite difference schemes.
To the authors' knowledge, the first conservative finite difference scheme 
that can be cast under this light is the 1D semi-Lagrangian 
solver created by J-M Qiu and Shu in 2011 \cite{QiuShu2011}.
There, the authors solve $q_{,t} + (a(t,x) q)_{,x} = 0$ with semi-Lagrangian
technology, which converts temporal integrals into spatial integrals.  Results
are extended to 2D incompressible flows by applying operator splitting techniques.  
In order to maintain mass conservation, they make use of a time-averaged version
of the usual discrete implicit flux function, which is not explicitly stated
in their work.  We elaborate and extend this idea in \S\ref{subsec:TI-1D}.

Our contribution is the development of time-averaged fluxes for conservative
finite difference methods, which include non-linear scalar, and
systems of equations for single and multidimensional systems of hyperbolic
conservation laws.  Furthermore, we introduce an entirely new interpretation
of time-averaged fluxes that can be applied to a large class of Runge-Kutta
methods.  Both the Taylor and Runge-Kutta interpretation only 
require a single projection onto characterstic variables per time step.
We make use of time-averaged discrete implicit flux
functions in order to retain mass conservation.  We refer the reader to the
references for details on the essentially non-oscillatory method
\cite{HaEnOsCh87,ShuOsher87,ShuOsher89}, as well as the weighted essentially
non-oscillatory method \cite{LiuOsherChan94,JiangShu96,Shu97,Shu09}.  
The Picard integral formulation is agnostic towards the choice of time
discretization, and therefore we begin by presenting the
conservative finite difference reconstruction procedure based upon
time-averaged fluxes.  Two examples of time discretizations are presented in
\S\S\ref{subsec:1D-Taylor}-\ref{subsec:1D-RK}, a stability analysis is
presented in \S\ref{sec:stability}, and the extension to 2D is presented 
in \S\ref{sec:modified-flux-2d}.

\section{The 1D Picard integral formulation}
\label{sec:PIF1d}

\subsection{1D conservative finite differences with time-averaged fluxes}
\label{subsec:TI-1D}

Consider a uniform grid of $m_x$ equidistant 
points in $\Omega = [a,b]$:
\begin{equation}
  x_i = a + \Par{i-\frac{1}{2}}\Delta x,
  \qquad \Delta x = \frac{b-a}{m_x},
  \qquad i\in\{1,\dots,m_x\}.
\end{equation}
In a finite difference method, one enforces some approximation to 
\eqref{eqn:1D_system.PDE} to hold on a finite set of points $x_i$, and seeks 
point value approximations $q_i^n \approx q(t^{n},x_i)$ to hold at discrete 
time levels $t^n$.

The classical conservative finite difference WENO
method implicitly defines a function $h( q(t, x) )$, whose sliding
average\footnote{The term `sliding average' was borrowed from the high-order
finite volume methods of Harten and collaborators \cite{HaEnOsCh87} before being imported
to conservative finite difference methods \cite{ShuOsher87,ShuOsher89}.} agrees
with $f(q(t,x))$ via
\begin{equation}\label{eqn:1D_system.TI-h-small}
  \bar{h}( q(t,x) ) := \frac{1}{\Delta x} \int_{x-\Delta x/2}^{x+\Delta x/2} h( q(t,s) )\, ds\ = f(q(t,x)).
\end{equation}
We consider a time-averaged version of this implicit function, defined by
\begin{equation}\label{eqn:h-avg}
  H^n( x ) := \frac{1}{\dt} \int_{t^n}^{t^{n+1}} {h}( q(t,x) )\, dt.
\end{equation}
After integrating \eqref{eqn:1D_system.TI-h-small} in time, we see
\begin{equation}\label{eqn:1D_system.TI-H}
  \h^n( x ) := \frac{1}{\dt} \int_{t^n}^{t^{n+1}} \bar{h}( q(t,x) )\, dt
  = \frac{1}{\Delta x} \int_{x-\Delta x/2}^{x+\Delta x/2} H^n( s )\, ds\ = F^n(x).
\end{equation}
We remark that \eqref{eqn:1D_system.TI-H} is similar to the usual averaging property found in 
finite difference methods, but {\bf the primary departure is that we have chosen to integrate
both the flux and the primitive of the flux function in time,}  
which differs from what is typically 
done in a finite difference WENO method, 
where the reconstruction procedure is carried out at discrete ``frozen in
time'' levels.   
We now point out two important properties concerning the implicitly defined 
function, which are in common with any classical MOL formulation:
\begin{itemize}
\item High-order reconstruction algorithms can produce high-order 
interface values $H^n( x_{i-1/2} )$ by identifying cell averages as known point
values: $\h^n (x_i) = F^n_i$, where $F^n_i := F^n( x_i )$, and therefore
$H^n$ need never be computed.
\item High-order derivatives of $F^n$ are computed by evaluating 
$H^n$ at interface points and applying the Fundamental Theorem of Calculus \cite{Shu97,Shu09}
\begin{equation}
\label{eqn:dF}
 \frac{d F^n}{dx}(x_i) = \frac{1}{\dx} \Bigl[ H^n(x_{i+1/2}) - H^n(x_{i-1/2}) \Bigr].
\end{equation}
\end{itemize}
Once \eqref{eqn:dF} has been defined, the complete finite difference
PIF-WENO method is given by inserting
the result into \eqref{eqn:picard1d}, which results in a single-step update 
\begin{equation}\label{eqn:1D_system.TI-cons}
  q_i^{n+1} = q_i^n -
    \frac{\dt}{\dx} \Par{ \hat{F}^n_{i+\frac{1}{2}} - \hat{F}^n_{i-\frac{1}{2}}},
\end{equation}
where $\hat{F}^n_{i-1/2} := H^n( x_{i-1/2} )$.
Before providing the description of how the $H^n_{i-1/2}$ are constructed, we 
point out a single remark followed by an important theorem.

\begin{rmk} \label{rmk:reduction-to-euler}
In the case where the time-averaged fluxes are approximated as constants through
$F^n( x ) \approx f( q(t^n, x ) )$, the PIF-WENO formulation reduces to a single
Euler step of a classical WENO method.
\end{rmk}

\begin{thm}
\label{thm:1dcons}
No matter how the quantities $F^n_{i}$ are approximated,
on an infinite or periodic domain,
the update defined by \eqref{eqn:1D_system.TI-cons}
satisfies the conservation property
\begin{equation}
  \sum_i q_i^{n+1} = \sum_i q_i^n.
\end{equation}
\end{thm}
\begin{proof}
Sum equation \eqref{eqn:1D_system.TI-cons} over all $i$.
\end{proof}

This theorem is a consequence of the deliberate choice of the order of operations 
used to carry out \eqref{eqn:picard1d}.  Its
significance means that the time-averaged fluxes can be discretized in any
manner without violating mass conservation.
Under the assumption that the $F^n_i$ have been approximated 
at each grid point, 
we are now ready to describe the finite difference WENO procedure for systems.
A Taylor discretization will be presented in \S \ref{subsec:1D-Taylor}, and a
Runge-Kutta discretization will be presented in \S \ref{subsec:1D-RK}.

\subsection{The finite difference WENO method with time-averaged fluxes}
\label{subsec:fd-weno-ta}

\smallskip

\begin{enumerate}

%
%
\item
Compute average values of $q$ at the half grid points
\begin{equation} \label{eqn:roe-avg}
   q^*_{i-1/2} = \frac{1}{2}\left( q_i + q_{i-1} \right).
\end{equation}
Roe averages \cite{Roe81} that satisfy
$f(q_i) - f( q_{i-1} ) = f'( q^*_{i-1/2} ) \left( q_i - q_{i-1} \right)$
may be used in place of \eqref{eqn:roe-avg}.
We use the simple arithmetic average.

\item
Compute the left and right eigenvalue decomposition
of $f'(q) = R \Lambda R^{-1}$ at the half-grid points
\begin{equation*}
   R_{i-1/2} = R( q^*_{i-1/2} ), \quad
   R^{-1}_{i-1/2} = R^{-1}( q^*_{i-1/2} ).
\end{equation*}
Compute the fastest (scalar) local wave speed
\begin{equation*}
    \alpha^*_{i-1/2} := \max\left\{ 
        \left|\, \min \left\{ \Lambda(q_{i-1}),\, \Lambda(q^*_{i-1/2}) \right\} \right|,
        \left|\, \max \left\{ \Lambda(q^*_{i-1/2} ),\, \Lambda(q_i) \right\} \right| \right\}.
\end{equation*}
Our definition of $\alpha^*_{i-1/2}$ is nominally different 
than other definitions that have been used for local wave speeds in 
finite difference WENO methods \cite{BaDiShu00,VuSo02}.
For stability, we follow the common practice of
increasing this speed by exactly 10\% and define
$\alpha_{i-1/2} := 1.1 \cdot \alpha^*_{i-1/2}$, which helps to guarantee that
the approximate wave speeds encompass the exact solution to a Riemann problem.

\smallskip

\item
For each $i$, determine the weighted ENO stencil  $\{i + r\}$ surrounding
$x_{i-1/2}$.
In fifth-order WENO, the full stencil is 
given by $r \in \{-3, -2, -1, 0,1,2 \}$.  Project each $q_{i+r}$ and each
\emph{time-averaged} flux value $F_{i+r}$ onto the characteristic variables 
using the linear mapping $R^{-1}_{i-1/2}$
\begin{subequations}
\label{eqn:ta-projections}
\begin{align}
   w_{i + r} = R^{-1}_{i-1/2} \cdot q_{i + r}, \\
   z_{i + r} = R^{-1}_{i-1/2} \cdot F_{i + r}.
\end{align}
\end{subequations}
Apply a (local) Lax-Friedrichs flux splitting to \eqref{eqn:ta-projections}
\begin{align}
   z^{\pm}_{i+r} = \frac{1}{2}\left( z_{i+r} \pm \alpha_{i-1/2} \,  w_{i+r} \right).
\end{align}
As an alternative, one could use a global wave speed.  

\smallskip

\item
Perform a WENO reconstruction on each of the characteristic variables
separately.
Use the stencil that uses an extra point on the upwind direction for 
defining $z^\pm$:
\begin{align*}
 \hat{z}^+_{i-1/2} &= 
 WENO5^+ \left[ z^+_{i-3}, z^+_{i-2}, z^+_{i-1}, z^+_{i}, z^+_{i+1} \right], \\
 \hat{z}^-_{i-1/2} &= 
 WENO5^- \left[ z^-_{i-2}, z^-_{i-1}, z^-_{i}, z^-_{i+1}, z^-_{i+2} \right].
\end{align*}
See Appendix \ref{app:weno-reconstruct} for a description of this function.
Define $\hat{z}_{i-1/2} := \hat{z}^+_{i-1/2}  + \hat{z}^-_{i-1/2}$.

\item
Using the same projection matrix, $R_{i-1/2}$, project characteristic variables 
back onto the conserved variables
\begin{equation*}
    \hat{F}_{i-1/2} := R_{i-1/2} \cdot \hat{z}_{i-1/2}.
\end{equation*}

\end{enumerate}


To summarize, given \emph{any} modified flux function $F^n$, we compute
$\hat{F}^n_{i-1/2}$ with WENO reconstructions applied to flux-split
characteristic decompositions.  The result is inserted into
\eqref{eqn:1D_system.TI-cons} to update the solution.  This procedure can be
carried out independent of the choice of how the temporal integrals are
discretized.  As pointed out in Remark \ref{rmk:reduction-to-euler}, this
formulation reduces to classical WENO with forward Euler time-discretization
if the time-averaged fluxes are approximated with constants.  
We present two methods of increasing the temporal accuracy: a Taylor and a
Runge-Kutta example.

\subsection{A Taylor series discretization of the 1D Picard integral formulation}
\label{subsec:1D-Taylor}


Consider a Taylor expansion of $f$ centered at $t=t^n$, and 
define an approximation to the time-averaged flux as
\begin{equation}\label{eqn:1D_system.TI-F}
  F_T^n(x) := f( q(t^n,x) ) 
    + \frac{\dt}{2!}   \frac{df}{dt} ( q(t^n,x) ) 
    + \frac{\dt^2}{3!} \frac{d^2f}{dt^2}( q(t^n,x) ) = F^n + \BigOh(\dt^3).
\end{equation}
%
%
These temporal derivatives can be computed via the Cauchy-Kowalewski procedure.
For example, in 1D the first two derivatives are
\begin{align}
\label{eqn:time-derivs}
   \frac{df}{dt}     &= \pd{f}{q} \cdot q_{,t} = - \pd{f}{q} \cdot f_{,x} \\
   \frac{d^2f}{dt^2} &= \pdn{2}{f}{q} \cdot \bigl(f_{,x}\,, f_{,x}\bigr) 
                      + \pd{f}{q} \cdot \left( \pd{f}{q} \cdot f_{,x} \right)_{,x} \\ \nonumber
                     &= \pdn{2}{f}{q} \cdot \bigl(f_{,x}\,, f_{,x}\bigr) 
                      + \pd{f}{q} \cdot \left( \pdn{2}{f}{q} \cdot \left(
                      q_{,x},\, f_{,x} \right) +  \pd{f}{q} \cdot f_{,xx} \right),
\end{align}
where $\pd{f}{q} \in \mathbb{R}^{m\times m}$ is the Jacobian matrix of $f$, 
and $\pdn{2}{f}{q} \in \mathbb{R}^{m\times m\times m}$ is the Hessian 
tensor\footnote{Given two vectors $u,v\in\R^m$, the Hessian product is defined
as the vector whose $i^{th}$-component is $\left[ \pdn{2}{f}{q} \cdot ( u, v )
\right]_i := \sum_{j=1,k=1}^m \frac{ \partial^2f_i }{ \partial q_j \partial
q_k} u_j v_k$.}.
Further derivatives produce tensors that grow exponentially in size.
We compute point values ${F}_i^n := {F_T}^n( x_i )$ 
using the following finite difference formulas on a single 5-point stencil:
\begin{subequations}
\label{eqn:u-deriv}
\begin{align}
\label{eqn:ux}
u_{i,\, x} &:= \frac{1}{12 \dx} \left( 
        u_{i-2} - 8  u_{i-1} 
      + 8  u_{i+1} - u_{i+2} 
  \right) =
  u_{\, ,\, x}( x_i ) + \BigOh\left( \dx^4 \right)
  \\
\label{eqn:uxx}
u_{i,\, xx} &:= \frac{1}{12 \dx^2} \left( 
      - u_{i-2} 
      + 16 u_{i-1} 
      - 30 u_{i}
      + 16 u_{i+1} - u_{i+2} 
  \right)
  = u_{\, ,\, xx}( x_i ) + \BigOh\left( \dx^4 \right).
\end{align}
\end{subequations}
Higher derivatives can be included using the same stencil.  They will start to
lose single orders of accuracy, but because those derivatives
are multiplied by increasing powers of $\dt = \BigOh(\dx)$ the overall order 
of accuracy is retained.  This is consistent with previous results in the
literature \cite{QiuShu03,QiuDumbserShu05,JiangShuZhang13,SeGuCh13}.
This defines a single value for each $F^n_i$, and therefore 
any stencil can be used without violating mass conservation.

It bears notice to compare the present Taylor discretization with similar
finite difference Lax-Wendroff discretizations for the same problem.
For example, similar to Jiang, Shu and Zhang's recent work
\cite{JiangShuZhang13}, we have taken care to expand the second spatial
derivative involving $\left( \pd{f}{q} \cdot f_{,x} \right)_{,x}$ because this
allows us to compute the first, second and possibly high-order 
derivatives on the same stencil by differentiating the same interpolating
polynomial.  Without expansion, one could compose a first-derivative finite
difference stencil with itself, but that increases the size of the
computational stencil.  Indeed, in Qiu and Shu's earlier work \cite{QiuShu03},
the authors perform compositions of finite differences that creates a
rather large effective stencil.  Moreover, their method requires the use of
central finite differences in order to retain mass conservation, whereas our
scheme could have used any approximation to the time-averaged flux.  The primary
difference that the Taylor discretization of the Picard integral formulation
carries is its generality.


In \S\ref{sec:modified-flux-2d} we follow the direct extension of this method to
the two-dimensional case, but first we present an alternative MOL type of
method that can be derived from the Picard integral formulation.

\subsection{A Runge-Kutta discretization of the 1D Picard integral formulation}
\label{subsec:1D-RK}

As an alternative to the single-stage, single-step multiderivative Taylor
method from the previous section, one could pursue a high-order collocation
type of method to discretize the time-averaged fluxes.  For example, an
application of the fourth-order Simpson's rule to \eqref{eqn:picard1d-b} 
would define
\begin{equation}
\label{eqn:simpson}
    F^n_{S}(x) := \frac{1}{6} \left[
        f( q(t^n,x) ) + 4 f( q(t^{n+1/2},x) ) + f( q( t^{n+1}, x ) )
    \right]
    = F^n + \BigOh( \dt^4 ),
\end{equation}
where $t^{n+1/2} = t^n + \dt/2$.  Given the prohibitively expensive cost of
this implicit scheme, we would prefer an explicit reinterpretation of
\eqref{eqn:simpson}.  For example, it is well known that when the classical
four-stage, fourth-order Runge-Kutta method (RK4) is applied to an ODE whose
right hand side depends on time only, then Simpson's rule and RK4 are
identical.  Therefore, we will interpret the final step in a
Runge-Kutta scheme as the approximation to an integral, which allows us to
derive a new class of Runge-Kutta type methods from the Picard integral
formulation.

We begin our discussion with a focus on the temporal discretization.  For
example, the RK4 method applied to \eqref{eqn:1D_system.PDE} defines the first
stage as $q^{(1)}(x) = q(t^n, x)$.  Further stages are given by 
\begin{subequations}
\label{eqn:rk4-stages}
\begin{align}
q^{(2)}(x) = q(t^n, x) - \frac{\dt}{2} \Par{ f( q^{(1)}(x) ) }_{,x}, \\
q^{(3)}(x) = q(t^n, x) - \frac{\dt}{2} \Par{ f( q^{(2)}(x) ) }_{,x}, \\
q^{(4)}(x) = q(t^n, x) - \dt           \Par{ f( q^{(3)}(x) ) }_{,x},
\end{align}
\end{subequations}
and the final update is given by
\begin{equation}
\label{eqn:rk4-update}
       q(t^{n+1},x) = q(t^n,x) - \frac{\dt}{6}\left[
           f\left( q^{(1)} \right)_{,x} + 2 \left( 
           f\left( q^{(2)} \right)_{,x} + 
           f\left( q^{(3)} \right)_{,x} \right) + 
           f\left( q^{(4)} \right)_{,x} 
           \right].
\end{equation}
Equations \eqref{eqn:rk4-stages} and \eqref{eqn:rk4-update} define the
classical MOL temporal discretization of finite difference, finite volume or
discontinuous Galerkin methods, where each continuous spatial differentiation
constitutes a linear (in some cases of DG), or non-linear (in the case of FD or
FV-WENO) discrete differentiation operator.  In any case, the Picard integral
formulation will allow us to view \eqref{eqn:rk4-update} from a new angle.

For example, in place of updating our solution with \eqref{eqn:rk4-update}, 
we define 
\begin{equation}
\label{eqn:ta-rk}
    F^n_{RK}( x_i ) := \frac{1}{6} \left[ f\left( q^{(1)}(x_i) \right) + 
    2 \left( f\left( q^{(2)}(x_i) \right) + f\left( q^{(3)}(x_i)\right)
    \right) + f\left( q^{(4)}(x_i) \right) \right]
\end{equation}
which is a high-order approximation to the exact time-averaged flux.  After
inserting \eqref{eqn:ta-rk} into \eqref{eqn:picard1d-a}, we apply the
conservative finite difference WENO method as described in
\S\ref{subsec:fd-weno-ta} to the result, which updates our solution after each
stage value has been computed.  It remains to define how each stage value is
computed.  

If each stage value in \eqref{eqn:rk4-stages}, were computed with the full
FD-WENO procedure, then the method described so far would be computationally
equivalent to applying a standard MOL discretization.  However, in the Picard
integral formulation, we have the leeway to modify the fluxes and therefore can
reduce the computational cost of the method.  For example, we find that for
each stage in Eqn. \eqref{eqn:rk4-stages}, there is no need to project onto
the characteristic variables in Step 3 of \S\ref{subsec:fd-weno-ta},
and therefore Steps 1,2 and 5 of \S\ref{subsec:fd-weno-ta} can be
skipped.  The end result is that {\bf we can construct a method that requires
only one characteristic projection per time step} in place of a total of four
characteristic projections.  This is a large savings for systems of equations,
given that matrices for each flux interface as well as the multiplication of
each of these matrices over the footprint of each interface value
can be avoided.  Additionally, one
could consider replacing the non-linear fifth-order reconstruction procedure
in Step 4 with other methods.  However, after numerical experimentation, we
have found that further limiting in the form of a non-linear reconstruction is
necessary in order to obtain non-oscillatory behavior for methods with
sufficiently large time steps.  

We repeat that many variations of the methods presented thus far can be
derived from the Picard integral formulation.  We simply advocate the use of
constructing discretizations for time-averaged fluxes in place of the usual
``frozen-in-time'' approximations.  Extensions to 2D are straightforward, 
and will be presented in \S\ref{sec:modified-flux-2d}.



\section{Linear stability analysis}
\label{sec:stability}

Linear stability analysis for finite difference WENO methods is difficult
given that the WENO reconstruction procedure is a \emph{non-linear} operator.
In this section, we follow the common convention of assuming the so-called
\emph{linear weights} \cite{WaSp07,MoMaRu11}, which are valid in smooth
regions. Additionally, we assume the ideal wave speed case where we do not
increase the value of $\alpha$ in Step 2 of \S\ref{subsec:fd-weno-ta}.

A typical stability analysis usually relies on the two step procedure: i) determine
the eigen-spectrum of the spatial discretization operator, followed by ii)
consider the test problem $y' = \lambda y$ and verify that the region of
absolute stability for the ODE solver contains each eigenvalue from the
spectrum of the spatial discretization operator.  At first glance, it is not
readily obvious that the RK formulations derived from the Picard integral
follow this same recipe.  However, it will be shown in this section that they
have identical stability regions to their MOL counterparts, under the
assumption of linear reconstruction weights.  However, the
Taylor formulations derived from the Picard integral formulation carry an
added difficulty which will be addressed at the end of this section.

\subsection{Linear stability analysis for the Runge-Kutta discretization}

In this section, we will show that Runge-Kutta methods constructed from the
Picard integral formulation have identical stability regions to their
classical MOL counterparts.  Without much added difficulty, we show the
generic case in place of focusing on fourth-order Runge-Kutta, and in turn,
this defines the extension of the method already presented in 
\S\ref{subsec:1D-RK}.

To begin, consider scalar advection,
\begin{equation}
\label{eqn:1dCCadv}
    q_{, t} + \left( u q \right)_{, x} = 0, \quad x \in \R,
\end{equation}
whose flux function is defined as
$f(q) = u q$, with $u \in \R$ being a constant.  
The application of a Runge-Kutta method in
a classical MOL formulation produces a final update of the form
\begin{equation} \label{eqn:rk-update-generic}
    q^{h, n+1} = q^{h,n} + \dt \sum_{j=1}^s b_j \L\left( q^{h,(j)} \right),
\end{equation}
where $q^{h,n} \approx q(t^n,x)$ is the current approximation to the vector of
conserved variables, 
$\L( q^h ) \approx -f(q)_{,x}$ is the discrete derivative of the flux
function, 
and $\{ b_j \}_{1\leq j \leq s}$ are
the coefficients from the Butcher tableau of the method.  
The RK discretization of the PIF formulation computes
stage values $q^{h, (j)}$ that are identical to those computed from the
MOL discretization.
This is because the RK discretization of PIF neglects characteristic
projections, and \eqref{eqn:1dCCadv} is
a scalar problem.
Therefore, the only possible difference would be the final update, which 
we now show is identical.

The proposed RK methods derived from the Picard integral formulation update 
the solution with 
\begin{equation}
\label{eqn:rk-update-pif}
    q^{h,n+1} = q^{h,n} - \dt \left( \sum_{j=1}^s b_j f( q^{h, (j)} ) \right)_{,x}
\end{equation}
in place of \eqref{eqn:rk-update-generic}.  The derivative in
\eqref{eqn:rk-update-pif} would normally be interpreted as a non-linear
operator, but under the assumption of linear weights, we may proceed further
with the analysis.
Now, consider the \emph{linear} discrete spatial discretization operator 
$A \approx \partial_x$ that approximates the derivative of the function.
After inserting $f(q) = u q$ into \eqref{eqn:rk-update-pif}, the discrete
update for the proposed class of methods becomes
\begin{equation}
\label{eqn:rk-linear-final}
    q^{h, n+1} = q^{h,n} - \dt A \left( \sum_{j=1}^s b_j u q^{h, (j)}     \right)
               = q^{h,n} - \dt   \left( \sum_{j=1}^s b_j A( u q^{h,(j)}) \right).
\end{equation}
This is identical to \eqref{eqn:rk-update-generic} after inserting
$\L( q^{h, (j)} ) = - A( u q^{h, (j)} )$, and therefore the two schemes have
identical stability limits.  The primary difference is that the second
equality in \eqref{eqn:rk-linear-final} is only valid under the assumption of
linear weights.  The non-linear WENO reconstruction procedure that
computes the derivatives of the solution will make the methods different in the
presence of shocks.  A complete analysis of the \emph{non-linear} method is
beyond the scope of the present work, but would follow according to recent
results \cite{WaSp07,MoMaRu11}, which first requires an analysis of the linear
weights.

\subsection{Linear stability analysis for the Taylor discretization}

The Taylor series methods that discretize time with the Lax-Wendroff
procedure require a deeper analysis.  To see the added difficulty,
recall that the semi-discrete version of \eqref{eqn:1dlinear-discrete} is
given by
\begin{equation}
\label{eqn:1dlinear-discrete}
   q^h_{,t} = A u q^h.
\end{equation}
If a third-order Taylor method were to be directly applied to the
semi-discrete ODE in \eqref{eqn:1dlinear-discrete}, then the update would look
like
\begin{equation}
\label{eqn:1d-linear-taylora}
q^{h, n+1} 
    = \left(I + \dt u A + \frac{( \dt u A )^2}{2!} + \frac{(\dt u A)^3}{3!}
    \right) q^{h, n}.
\end{equation}
However, when performing the Lax-Wendroff procedure, higher derivatives do not
come from \eqref{eqn:1dlinear-discrete}, but from approximating higher
derivatives of the PDE.  Moreover, the Taylor discretization of the Picard
integral formation operates on the fluxes, and not on the conserved variables.
To see the difference, we factor out a single $\dt u A$ 
term from \eqref{eqn:1d-linear-taylora}, and rewrite it as
\begin{subequations}
\begin{align}
\label{eqn:linear-stab-a}
q^{h, n+1} 
    &= q^{h, n} + \dt u A \left(I + \frac{\dt u A}{2!} + \frac{(\dt u
    A)^2}{3!} \right) q^{h, n}  \\
\label{eqn:linear-stab-b}
    &\approx q^n - \dt u \partial_x \left(I - \frac{\dt u \partial_x }{2!} +
    \frac{(\dt u \partial_x)^2}{3!} \right) q^{ n}.
\end{align}
\end{subequations}
The outermost derivative is approximated with the difference operator $A$, but
the interior derivatives are approximated with Equation
\eqref{eqn:u-deriv} as opposed to
composing the conservative reconstruction procedure with itself multiple
times. The purpose of doing this was to retain a compact stencil.  As is common with all
Lax-Wendroff type of methods, we cannot simply study the eigen-spectrum of 
\eqref{eqn:1dlinear-discrete} in our analysis, and we therefore turn to a
fully discrete von Neumann analysis.


To begin, we assume the solution to Eqn. \ref{eqn:1dCCadv} has the form $q^n_i
= e^{ J k i \dx }$, where $J = \sqrt{-1}$, $i \in \Z$, $\dx \in \R_{>0}$ is
the grid spacing, and $k \in \R$ is a a single wave number \cite{book:Le02}.
If we insert this into our Taylor method, a single update can be written as
$q^{n+1}_i = g_T(k, \dx, \dt) q^n_i$, where the amplification 
factor
is given by 
\begin{equation}
\begin{split}
&g_T( k, \dx, \dt ) := \\
1 &+ \nu \left( \frac{1}{20} e^{2 J k \dx} - \frac{1}{2} e^{J k \dx} -
\frac{1}{3} + e^{-J k \dx} - \frac{1}{4} e^{-2 J k \dx} + \frac{1}{30} e^{-3 J
k \dx} \right)  \\
&+ \nu^2 \left( \frac{1}{480} e^{4 J k \dx} - \frac{3}{80} e^{3 J k \dx} +
\frac{11}{72} e^{2 J k \dx} + \frac{61}{360} e^{J k \dx} - \frac{41}{80}
\right. \\
& \left.-
\frac{1}{180} e^{-J k \dx} + \frac{121}{360} e^{-2 J k \dx} 
- \frac{1}{8} e^{-3 J k \dx} + 
\frac{31}{1440} e^{-4 J k \dx} - \frac{1}{720} e^{-5 J k \dx} \right)  \\
&+ \nu^3 \left(
 -\frac{1}{1440} e^{4 J k \dx} + \frac{13}{720} e^{3 J k \dx} - \frac{55}{432}
 e^{2 J k \dx} + \frac{71}{540} e^{J k \dx } + \frac{91}{360} \right. \\
& \left.-
 \frac{583}{1080} e^{-J k \dx} + \frac{731}{2160} e^{-2 J k \dx} -
 \frac{1}{12} e^{-3 J k \dx} + \frac{47}{4320} e^{-4 J k \dx} - \frac{1}{2160}
 e^{-5 J k \dx} \right),
\end{split}
\end{equation}
and the CFL number is defined as $\nu = u \dt / \dx$.  Details on deriving the
amplification factor were computed using the open-source symbolic toolkit SymPy and have
been omitted for brevity.  The method is stable for a given time step $\dt$
and grid spacing $\dx$ provided that $|g_T( k, \dx, \dt )| \leq 1$ for all
wave numbers $k$.  

Note that the amplification factor is a polynomial in $\nu$ and $\rho := e^{J
k \dx}$.  The degree of $\nu$ represents the temporal order, and the degree of
$\rho$ represents the footprint of the stencil.  Indeed, the fourth-order
Taylor version of our method would contain a term involving $\nu^4$, and
assuming we restrict our attention to the identical stencil used to compute
each derivative in Eqn. \ref{eqn:u-deriv}, the polynomial multiplying $\nu^4$
will have terms ranging from $\rho^4$ down to $\rho^{-5}$.

In light of attempting to find the roots of $g_T$, in Fig.~\ref{fig:vn} we plot
the absolute value of $g_T$ as a function of the CFL number, and 
wave number.  

\begin{figure}
\begin{center}
\includegraphics[width=0.4\textwidth]{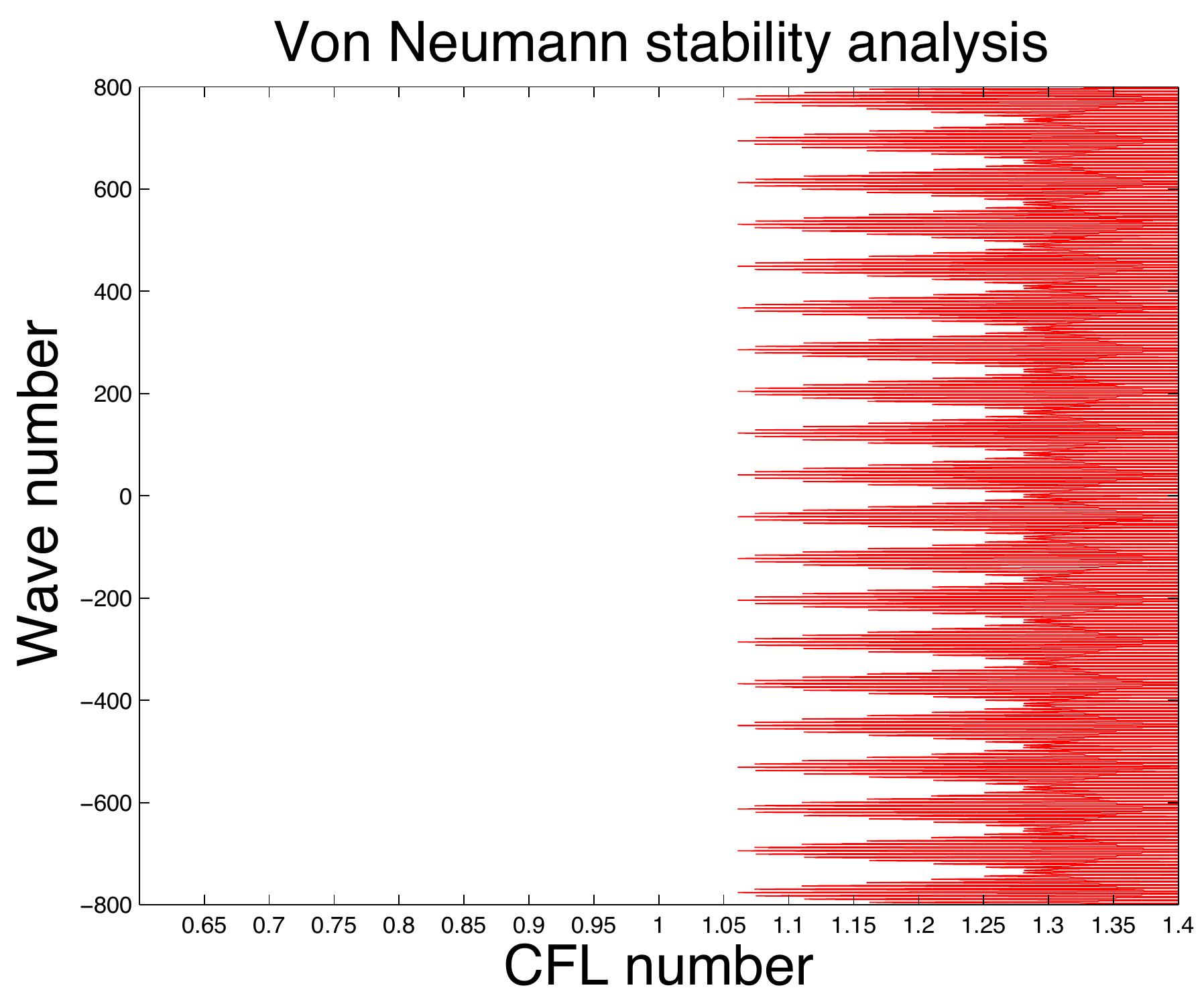}
\caption{Von Neumann stability analysis for Taylor discretization.  
Shown here are the contour lines for $|g_T| = 1$.  Without loss of generality,
we set $\dx = 1$.  Regions to the left of the
boundary are stable, and regions to the right of the region are unstable.
We observe that the largest stable time step for the third-order Taylor 
discretization that is valid for all wave numbers is slightly larger than 1.0606.
\label{fig:vn}
}
\end{center}
\end{figure}

\section{The 2D Picard integral formulation}
\label{sec:modified-flux-2d}

A 2D conservation law is given by
\begin{equation}
\label{eqn:2dhyper}
    q_{,t} + \left( f(q) \right)_{,x} 
           + \left( g(q) \right)_{,y} = 0.
\end{equation}
%
Similar to the 1D case, we define the 2D \emph{Picard integral formulation} as
\begin{subequations}
\begin{equation}
    q^{n+1} = q^n - \dt \left( F^n( x, y ) \right)_{, x} - \dt \left( G^n( x, y ) \right)_{, y},
\end{equation}
where the \emph{time-averaged fluxes} are defined by
\begin{equation}\label{eqn:2D_system.TI-MF}
    {F}^n( x, y) := \frac{1}{\dt} \int_{t^n}^{t^{n+1}} f(q(t,x,y))\, dt, \quad 
    {G}^n( x, y) := \frac{1}{\dt} \int_{t^n}^{t^{n+1}} g(q(t,x,y))\, dt.
\end{equation}
\end{subequations}

We again propose an application of the finite difference WENO method to the
time-averaged fluxes in place of the usual ``frozen in time'' approximation
to the physical fluxes.
We repeat that many options for discretizing \eqref{eqn:2D_system.TI-MF}
exist.  We begin with a description of how the 2D finite difference WENO
method can be applied under the assumption that one such time-averaged 
approximation has been made.

\subsection{2D conservative finite differences with time-averaged fluxes}
\label{subsec:TI-2D}

We start with a discretization of the domain $\Omega = [a_x, b_x] \times [a_y, b_y] \subset \R^2$ given by
\begin{subequations}
\begin{align}
  x_i &= a_x + \Par{i-\frac{1}{2}}\Delta x,
  \qquad &\Delta x = \frac{b_x-a_x}{m_x},
  \qquad &i\in\{1,\dots,m_x\}, && \quad \\
  y_j &= a_y + \Par{j-\frac{1}{2}}\Delta y,
  \qquad &\Delta y = \frac{b_y-a_y}{m_y},
  \qquad &j\in\{1,\dots,m_y\}.&
\end{align}
\end{subequations}
The finite difference method seeks pointwise approximations
$q^n_{ij} \approx q\left( t^n, x_i, y_j \right)$ to hold at each of these grid
points.

In the hyperbolic 2D case, one has two eigen-decompositions for $f'(q)$ and $g'(q)$, to
which the procedure described in \S \ref{subsec:TI-1D} is applied on a
dimension by dimension fashion.  
In place of the instantaneous point-wise values such as
$f(t^n,x_i,y_j)$ and $g(t^n,x_i,y_j)$,
we propose using suitable approximations of $F_{ij}^n$ and $G_{ij}^n$ to the time-averaged
fluxes in \eqref{eqn:2D_system.TI-MF}.  That is, we will define an update of
the form
\begin{equation}
\label{eqn:euler2d}
    q^{n+1}_{ij} = q^n_{ij}
        - \frac{\dt}{\dx} \left( \hat{F}^n_{i+1/2,j} - \hat{F}^n_{i-1/2,j}  \right)
        - \frac{\dt}{\dy} \left( \hat{G}^n_{i,j+1/2} - \hat{G}^n_{i,j-1/2} \right),
\end{equation}
where the differences in $\hat{F}^n_{i\pm 1/2, j}$ and 
$\hat{G}^n_{i, j\pm 1/2}$ approximate discrete derivatives of $F^n$ and $G^n$,
respectively, via the non-linear WENO reconstruction procedure.
Similar to 
Remark \ref{rmk:reduction-to-euler}, we point out the following special case.

\begin{rmk}
If $F^n_{ij} = f( q(t^n,x_i,y_j) )$ and $G^n_{ij} = f( q(t^n,x_i,y_j) )$,
then \eqref{eqn:euler2d} is identical to a single Euler step of a classical
RK-WENO method.  It is the inclusion of higher-order terms in the
approximation of the time-averaged fluxes that increases the temporal order of
accuracy.
\end{rmk}

\begin{rmk} This method is automatically mass conservative given \emph{any}
time-averaged flux.  This follows from the 2D extension of Theorm \ref{thm:1dcons}.
\end{rmk}

We repeat that many options exist for constructing the time-averaged fluxes.
At present, we offer Taylor methods and Runge-Kutta methods; additional
methods are currently under investigation, and will be presented in a follow
up work.

\subsection{A Taylor series discretization of the 2D Picard integral formulation}

We define Taylor discretizations of the time-averaged fluxes with
\begin{subequations}
\begin{align}
\label{eqn:2D_system.TI-F}
  F_T^n(x,y) := f( q(t^n,x,y) ) 
    + \frac{\dt}{2!}   \frac{df}{dt} ( q(t^n,x,y) ) 
    + \frac{\dt^2}{3!} \frac{d^2f}{dt^2}( q(t^n,x,y) ) = F^n + \BigOh(\dt^3); \\
\label{eqn:2D_system.TI-G}
  G_T^n(x,y) := g( q(t^n,x,y) ) 
    + \frac{\dt}{2!}   \frac{dg}{dt} ( q(t^n,x,y) ) 
    + \frac{\dt^2}{3!} \frac{d^2g}{dt^2}( q(t^n,x,y) ) = G^n + \BigOh(\dt^3).
\end{align}
\end{subequations}
Similar to the 1D case, analytical temporal derivatives can be computed via the 
Cauchy-Kowalewski procedure.
Derivatives of the first component in the flux function $f$ are 
given by
\begin{align}
\label{eqn:time-derivs-2df}
   \frac{df}{dt}     &= -\pd{f}{q} \cdot \left( f_{\!,\,x} + g_{\!,\,y} \right),  \\
   \frac{d^2f}{dt^2} &= \pdn{2}{f}{q} \cdot \bigl( f_{,x} + g_{,y}\,, f_{,x} + g_{,y} \bigr) 
                      - \pd{f}{q} \cdot \left( f_{,x} + g_{,y} \right)_{,t},
\end{align}
where an expansion of the final time derivative is given by
\begin{align}
\label{eqn:temporal-derivs-2dfpg}
\left( f_{,x} + g_{,y} \right)_{,t} =    
&  \pdn{2}{f}{q} \cdot \left( q_{,x}, f_{,x} + g_{,y} \right) + \pd{f}{q} \cdot \left( f_{,xx} + g_{,xy} \right)
+ \\ \nonumber 
&  \pdn{2}{g}{q} \cdot \left( q_{,y}, f_{,x} + g_{,y} \right) + \pd{g}{q} \cdot \left( f_{,xy} + g_{,yy} \right).
\end{align}
Note that third-order accuracy requires an approximation to the mixed
derivatives $f_{,xy}$ and $g_{,xy}$. 
%
Similarly, the temporal derivatives of $g$ can be expressed with
\begin{align}
\label{eqn:time-derivs-2dg}
   \frac{dg}{dt}     &= -\pd{g}{q} \cdot \left( f_{\!,\,x} + g_{\!,\,y} \right),  \\
   \frac{d^2g}{dt^2} &= \pdn{2}{g}{q} \cdot \bigl( f_{,x} + g_{,y}\,, f_{,x} + g_{,y} \bigr) 
                      - \pd{g}{q} \cdot \left( f_{,x} + g_{,y} \right)_{,t}.
\end{align}

We compute the first- and second-derivatives of $f, g$ and $q$ 
on a dimension by dimension fashion, with the coefficients prescribed by 
equations \eqref{eqn:ux}-\eqref{eqn:uxx}.  
{\bf The singular departure from the 1D case is that 
we require an approximation for the cross-derivative terms.}  These could be
computed through a composition of finite differences or inclusion of off
diagonal terms in the stencil.  In order to retain a compact stencil, we
choose the latter approach, to which we apply the finite difference formula 
\begin{equation}
u_{,\, xy}\left( x_i, y_j \right) \approx u_{ij,\, xy} := \frac{1}{4 \dx \dy } 
    \left(  
        u_{i+1,\, j+1} - u_{i-1,\, j+1} - u_{i+1,\, j-1} + u_{i-1,\, j-1}
    \right). 
\end{equation}
Given that they are eventually multiplied by a factor of 
$\dt^3 = \BigOh(\dx^3)$, we only compute these to second-order.
%
In the past, these cross-derivatives have been computed
by applying a single finite difference stencil multiple times \cite{QiuShu03},
which results in a large effective stencil.

The complete scheme is given by inserting
\begin{equation}
F^n_{ij} := F_T^n( x_i, y_j ) \quad \text{and} \quad
G^n_{ij} := G_T^n( x_i, y_j )
\end{equation}
into the conservative finite difference reconstruction procedure.

This current proposal provides an additional departure from previous finite difference
Lax-Wendroff works \cite{QiuShu03,JiangShuZhang13}, because
we use a single inexpensive stencil to compute the time-averaged fluxes
everywhere in the domain, and then follow that with a
single WENO reconstruction.
Moreover, 
from the 2D extension of Theorem \ref{thm:1dcons},
further limiting in the form of non-conservative WENO differentiation is an option
for us to compute higher derivatives of the flux function.
The overall effective stencil is presented in Fig.~\ref{fig:stencil}, and is much smaller than
Qiu and Shu's original method \cite{QiuShu03}, and similar in size but different
in shape to Jiang Shu and Zhang's recent proposal \cite{JiangShuZhang13}.
The shaded cells indicate the size of the stencil used in the first pass, and
the red `X' terms indicate the stencil required for the WENO reconstruction.
An implementation that uses a single loop over the domain would require access
to every circled element to update the solution at one point.

\subsubsection{A brief note on implementing the Taylor method}

One simple implementation of the proposed Taylor method 
that requires minimal modification to existing codes 
uses two loops over the computational domain:
\begin{enumerate}
    \item Compute the time-averaged fluxes $F^n_{ij}$ and $G^n_{ij}$ with finite
    differences on the conserved variables.  The benefit is that
    no characteristic projections or non-linear reconstructions are necessary.
    However, Taylor methods in particular require additional algebraic
    operations and therefore the
    computational expense comes from the need to evaluate high-derivatives
    using Jacobians and Hessians that increase in size.
    \item Use the time-averaged fluxes $F^n_{ij}$ and $G^n_{ij}$
    to perform the usual WENO reconstruction with an existing sub-routine.
    The cost of this step is equivalent to that of a single WENO
    reconstruction.
\end{enumerate}
We remark that the storage required for this particular implementation of the
Taylor method is competitive with low-storage Runge-Kutta methods
\cite{Williamson80,Ke08,Ke10,NieDieBu12}.  In particular, any three-stage
Runge-Kutta method applied to the classical MOL formulation will require a
total of three WENO reconstructions, whereas we make a first pass with an
inexpensive linear finite difference stencil, and then follow that up with a
single WENO reconstruction.  In addition, because it is a single-step Taylor
method, this method is amenable to adding adaptive mesh refinement, which is a
non-trivial task for multistage methods.

\begin{figure}
\begin{center}
\includegraphics[width=0.6\textwidth]{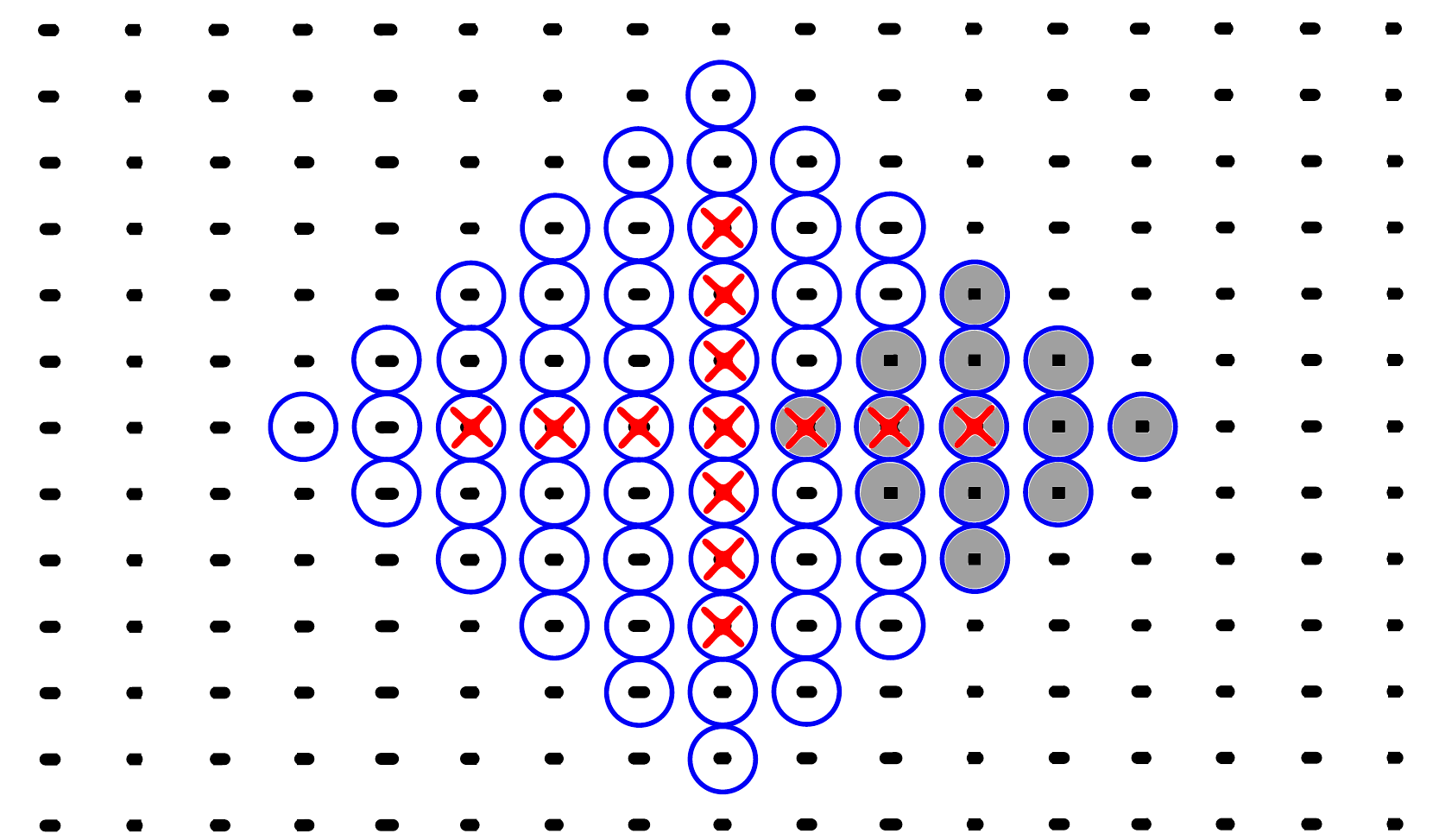}
\caption{Effective stencil for the Taylor method.  Shown here is the 
collection of all elements required to define the new solution,
$q^{n+1}_{ij}$, at a single point $(x_i, y_j)$.
Each red `X' comes from the 7-point fifth-order WENO stencil used in each
direction, where a 13-point stencil is used to define each required
$F^n_{i+s,j}$ and $G^n_{i,j+s}$.  For example, the elements required to define
$F^n_{i+3,j}$ have been shaded.  Note that the cross-derivative terms have
been explicitly included in the stencil as opposed to relying on many
compositions of the WENO stencil to produce these.  \label{fig:stencil}
}
\end{center}
\end{figure}

\subsection{A Runge-Kutta discretization of the 2D Picard integral formulation}

The extension of the RK formulated PIF-WENO method to 2D is straightforward.
For example, the 2D analogue of Eqn. \eqref{eqn:ta-rk} 
from \S\ref{subsec:1D-RK} 
is to define the time-averaged fluxes as
\begin{subequations}
\begin{align}
\label{eqn:ta-rk-2D}
  F_{RK}^n(x_i,y_j) 
      &:= \frac{1}{6} \left[ f\left( q^{(1)}_{ij} \right) + 
        2 \left( f\left( q^{(2)}_{ij} \right) + f\left( q^{(3)}_{ij}\right) \right) +
        f\left( q^{(4)}_{ij} \right) \right], \\
  G_{RK}^n(x_i,y_j) 
      &:= \frac{1}{6} \left[ g\left( q^{(1)}_{ij} \right) + 
        2 \left( g\left( q^{(2)}_{ij} \right) + g\left( q^{(3)}_{ij}\right) \right) +
        g\left( q^{(4)}_{ij} \right) \right],
\end{align}
\end{subequations}
which are high-order approximations to the exact time-integrated flux
functions, $F^n$ and $G^n$ in \eqref{eqn:2D_system.TI-MF}.
Again, the values $q_{ij}^{(k)}$, for $k=1, \dots 4$ are computed from the RK
stages, which do not require characteristic projections.  {\bf Identical to the 1D
case, the complete WENO reconstruction procedure is only applied to the final
stage.}  Similar to classical finite difference WENO methods, this
reconstruction procedure is applied in a dimension by dimension fashion, which
requires projections onto the characteristic variables.  The difference here
is that we operate on time-averaged fluxes, $F^n_{RK}$ and $G^n_{RK}$, as
opposed to ``frozen in time'' flux values.  

When comparing the two methods presented in this work, the primary difference
the RK method has from the Taylor method is that the integral for the
time-averaged flux is computed
through a collocation type of procedure, which requires computing approximate
stage value for the flux functions.  Again, we find that it is not necessary
to perform characteristic projections of the fluxes and conserved variables
for each intermediate stage value.  

\section{Numerical Results}
\label{sec:numerical-results}

\subsection{1D Burger's equation}
\label{subsec:burgers}

We begin our numerical results section with a convergence study on smooth
solutions to inviscid Burger's equation:
\begin{equation}
	q_{,t} +  \left( \frac{1}{2} q^2 \right)_{,x} = 0.
\end{equation}
We take the computational domain to be $[0,2]$ and apply periodic boundary
conditions.  The initial conditions are prescribed by
$q(t=0, x) = 0.5 + \sin\left( \pi x \right)$. 
For our convergence study, 
we choose a final time of $t=0.5 \pi^{-1}$ before any shocks develop.

Exact solutions to this problem can be found implicitly by applying 
the method of characteristics \cite{book:Le02}.
With initial conditions defined by $q_0( s ) = q(0, s )$, 
the exact solution is 
\begin{align} \label{eqn:burger-implicit}
	q(t,x) = q_0\left( x - t \cdot q_0( \xi ) \right),
\end{align}
where
$\xi$ is an implicit solution to $\xi = x - t \cdot q_0(\xi)$.  
To construct exact solutions, we solve \eqref{eqn:burger-implicit}
at each grid point using Newton iteration with a tolerance of $10^{-15}$.

In 1D, we use relative errors with the $L^1$ norm defined by 
point-wise values
\begin{equation}
\label{eqn:error1d}
   \text{Error } := 
   \frac{ \dx \sum_{i=1}^{m_x} \left|q^n_i - q(t^n, x_i) \right| } {
          \dx \sum_{i=1}^{m_x} |q(t^n, x_i)|  },
\end{equation}
and we define the CFL number as the dimensionless quantity
\begin{equation}
\label{eqn:cfl1d}
    \nu := \frac{\dt}{\dx} \max_i \left\{ \alpha^*_{i-1/2} \right\}.
\end{equation}
A convergence study for the Taylor discretization of PIF-WENO is presented in
Table \ref{table:burger-convergence}, and a convergence study for the RK 
discretization of PIF-WENO is presented in Table \ref{table:burger-convergence-rk}.

In Jiang, Shu and Zhang \cite{JiangShuZhang13}, the authors compare 
results of their finite difference WENO method in a side-by-side table 
against the older Qiu and Shu method \cite{QiuShu03}.  However, the 
new simulations use a smaller CFL number of
$\nu = 0.3$ as opposed to the value of $\nu = 0.5$ reported in 
the earlier paper. For
completeness, in Table \ref{table:burger-convergence} we present results 
of our method for both CFL numbers. We make two brief observations:
\begin{itemize}
\item Large CFL numbers bring out the temporal error faster, and therefore
high-order in time is likely more important to reduce error if a user 
wishes to take large time steps.
\item For coarse grids, the difference in errors between the two CFL numbers
is quite small, indicating that in this case, high-order spatial 
accuracy is more important than a high temporal order of accuracy.
\end{itemize}
Similar to past investigations,
for coarse meshes the spatial error dominates, and therefore the algebraic
order is higher than expected.  This practice is commonplace in the
WENO literature, where authors often use lower order time-integrators such as
SSP-RK3 
\cite{ShuOsher87,GoShu98}, and present results
indicating fifth or higher order convergence rates.  We observe this
with the Runge-Kutta discretization presented in Table \ref{table:burger-convergence-rk}.
In this work, we refine our mesh until the temporal error is observable.  
\begin{table}
\centering
\normalsize
\caption{Convergence study of Taylor method for Burger's equation.
The number of mesh elements $m_x$ is given in the first column, and 
errors are defined in \eqref{eqn:error1d}.
Two different CFL numbers of $\nu = 0.3$ and $\nu = 0.5$ as 
defined in \eqref{eqn:cfl1d} have been chosen in order to offer a comparison 
with previous Lax-Wendroff based finite difference WENO methods
\cite{QiuShu03,JiangShuZhang13}.
The `Order' columns refer to the algebraic order of convergence, computed as 
the base-2 logarithm of the ratio of two successive error norms.
\label{table:burger-convergence} }

\begin{tabular}{|r||c|c||c|c|}
\hline
\bf{Mesh} 
& \bf{Error ($\nu = 0.3$) } & \bf{Order}
& \bf{Error ($\nu = 0.5$) } & \bf{Order} \\
\hline
\hline
$  10$ & $1.94\times 10^{-02}$ & --- & $2.05\times 10^{-02}$ & ---\\
\hline
$  20$ & $3.40\times 10^{-03}$ & $2.51$ & $3.61\times 10^{-03}$ & $2.51$\\
\hline
$  40$ & $3.57\times 10^{-04}$ & $3.25$ & $3.81\times 10^{-04}$ & $3.25$\\
\hline
$  80$ & $1.72\times 10^{-05}$ & $4.38$ & $1.97\times 10^{-05}$ & $4.27$\\
\hline
$ 160$ & $6.13\times 10^{-07}$ & $4.81$ & $1.24\times 10^{-06}$ & $3.99$\\
\hline
$ 320$ & $3.03\times 10^{-08}$ & $4.34$ & $1.17\times 10^{-07}$ & $3.41$\\
\hline
$ 640$ & $3.12\times 10^{-09}$ & $3.28$ & $1.42\times 10^{-08}$ & $3.04$\\
\hline
$1280$ & $3.83\times 10^{-10}$ & $3.02$ & $1.78\times 10^{-09}$ & $3.00$\\
\hline
$2560$ & $4.81\times 10^{-11}$ & $2.99$ & $2.23\times 10^{-10}$ & $3.00$\\
\hline
$5120$ & $6.02\times 10^{-12}$ & $3.00$ & $2.79\times 10^{-11}$ & $3.00$\\
\hline
\end{tabular}
\end{table}

\begin{table}
\centering
\normalsize
\caption{Convergence study of RK method for Burger's equation.
The number of mesh elements $m_x$ is given in the first column, and 
errors are defined in \eqref{eqn:error1d}.
Two different CFL numbers of $\nu = 0.5$ and $\nu = 0.8$ have been chosen 
in an effort to expose the temporal error.  We observe that for small CFL numbers
the spatial error dominates the error.
\label{table:burger-convergence-rk} }
\begin{tabular}{|r||c|c||c|c|}
\hline
\bf{Mesh} 
& \bf{Error ($\nu = 0.4$) } & \bf{Order}
& \bf{Error ($\nu = 0.8$) } & \bf{Order} \\
\hline
\hline
$  10$ & $3.85\times 10^{-02}$ & --- & $3.85\times 10^{-02}$ & ---\\
\hline
$  20$ & $3.42\times 10^{-03}$ & $3.49$ & $3.76\times 10^{-03}$ & $3.36$\\
\hline
$  40$ & $3.68\times 10^{-04}$ & $3.22$ & $4.06\times 10^{-04}$ & $3.21$\\
\hline
$  80$ & $1.78\times 10^{-05}$ & $4.37$ & $1.88\times 10^{-05}$ & $4.44$\\
\hline
$ 160$ & $5.62\times 10^{-07}$ & $4.99$ & $5.94\times 10^{-07}$ & $4.98$\\
\hline
$ 320$ & $1.21\times 10^{-08}$ & $5.53$ & $1.37\times 10^{-08}$ & $5.44$\\
\hline
$ 640$ & $2.49\times 10^{-10}$ & $5.61$ & $3.48\times 10^{-10}$ & $5.30$\\
\hline
$1280$ & $4.89\times 10^{-12}$ & $5.67$ & $1.15\times 10^{-11}$ & $4.93$\\
\hline
$2560$ & $1.35\times 10^{-13}$ & $5.18$ & $5.79\times 10^{-13}$ & $4.31$\\
\hline
$5120$ & $7.09\times 10^{-15}$ & $4.25$ & $3.44\times 10^{-14}$ & $4.07$\\
\hline
\end{tabular}
\end{table}

\subsection{The 1D Euler equations}
\label{subsec:euler}
The Euler equations describe the evolution of density $\rho$, momentum 
$\rho u$ and energy $\E$ of an ideal gas:
\begin{equation}
\left( 
\begin{array}{c}
    \rho \\ \rho u \\ \E
    \end{array}
    \right)_{, t}
    + 
    \left( 
    \begin{array}{c}
       \rho u \\ \rho u^2 + p \\ ( \E + p ) u
    \end{array}
    \right)_{, x}
    = 0,
\end{equation}
where $p$ is the pressure.
The energy $\E$ is related to the primitive variables $\rho$, $u$ and $p$ by
$\E = \frac{p}{ \gamma-1 } + \frac{1}{2}\rho u^2$, and
$\gamma$ is the ratio of specific heats.  For all of our simulations,
we take $\gamma = 1.4$.

\subsubsection{The 1D Euler equations: a Riemann problem}

We present a difficult test case.  This is commonly 
referred to as the \emph{Lax shock tube},
however, the initial conditions are those defined by 
Harten \cite{Harten78,HaEnOsCh87}:
\begin{equation}
\label{eqn:shock-harten}
   (\rho, \rho u, \E )^T = 
   \begin{cases}
   (0.445, 0.3111, 8.928)^T, \quad \text{if } x \leq 0.5, \\
   (0.5, 0, 1.4275)^T, \quad \text{otherwise}.
   \end{cases}
\end{equation}

We select $t = 0.16$ for the final time of our simulation 
\cite{QiuShu03}, and we use a computational domain of $[0,1]$
with outflow boundary conditions.
Results for our method applied to this problem are presented in 
Fig.~\ref{fig:euler-shock-tube-harten-weno}.

\begin{figure}[!htb]
\centering
\includegraphics[width=0.3\textwidth]{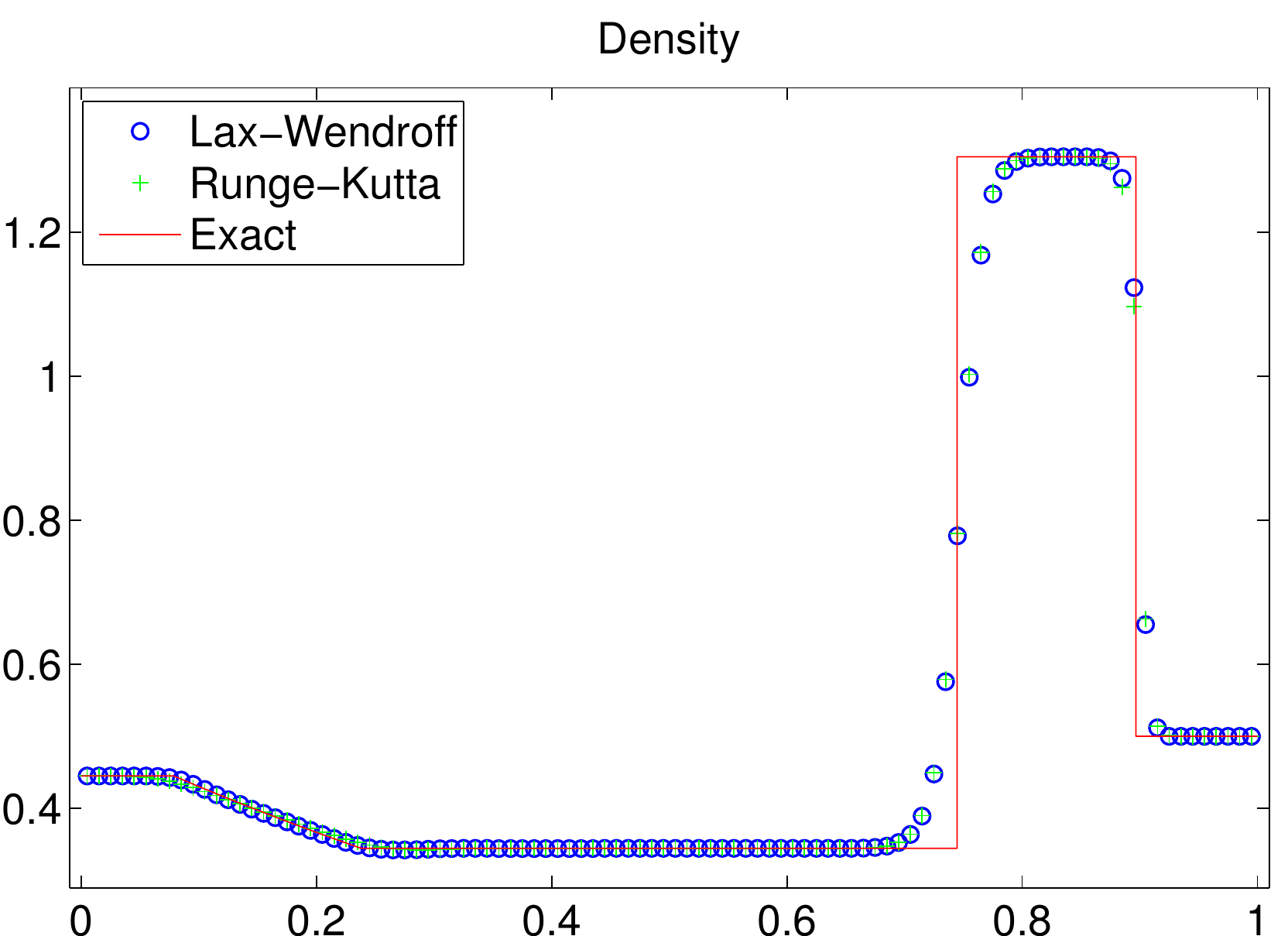}
\includegraphics[width=0.3\textwidth]{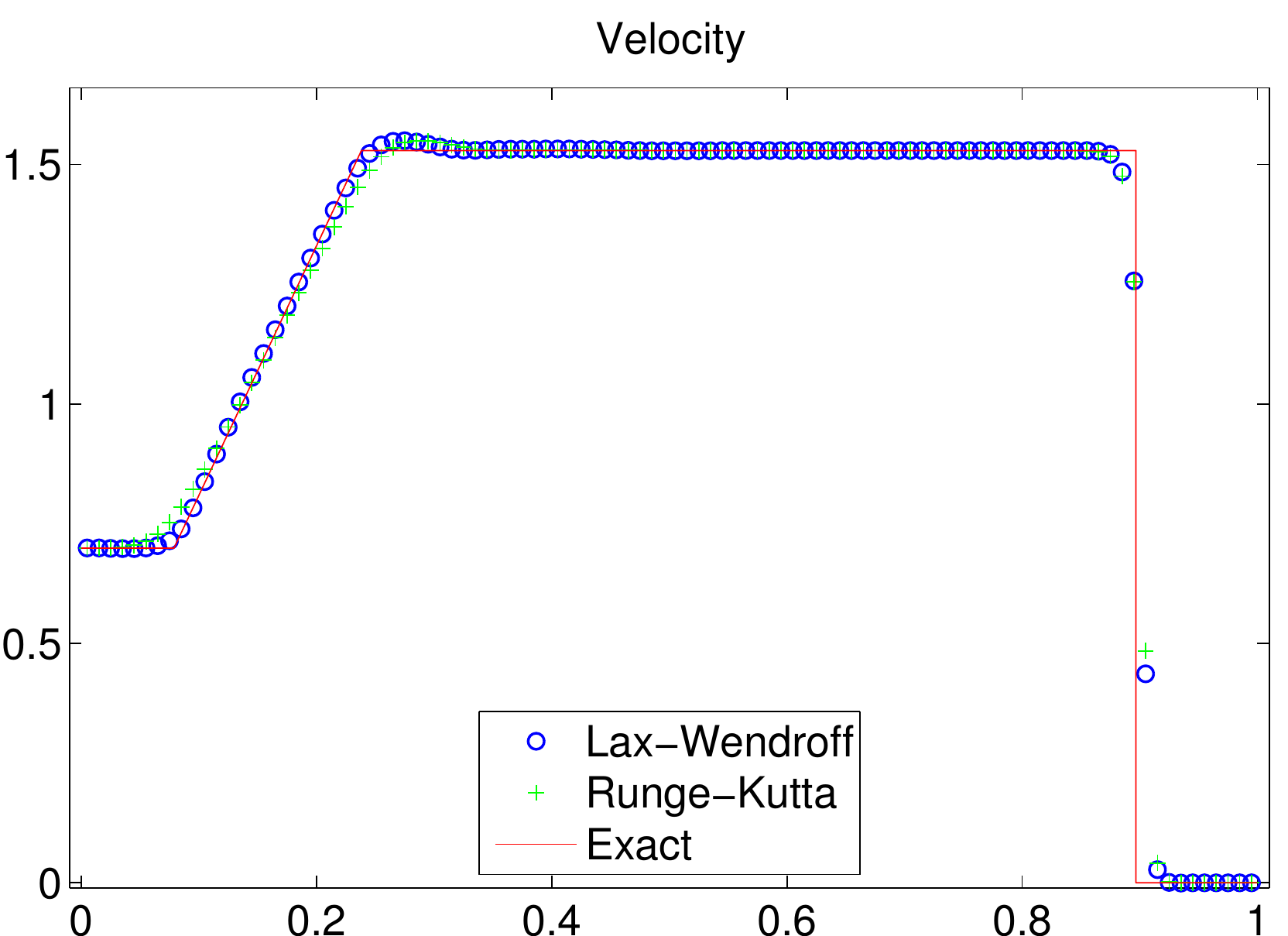}
\includegraphics[width=0.3\textwidth]{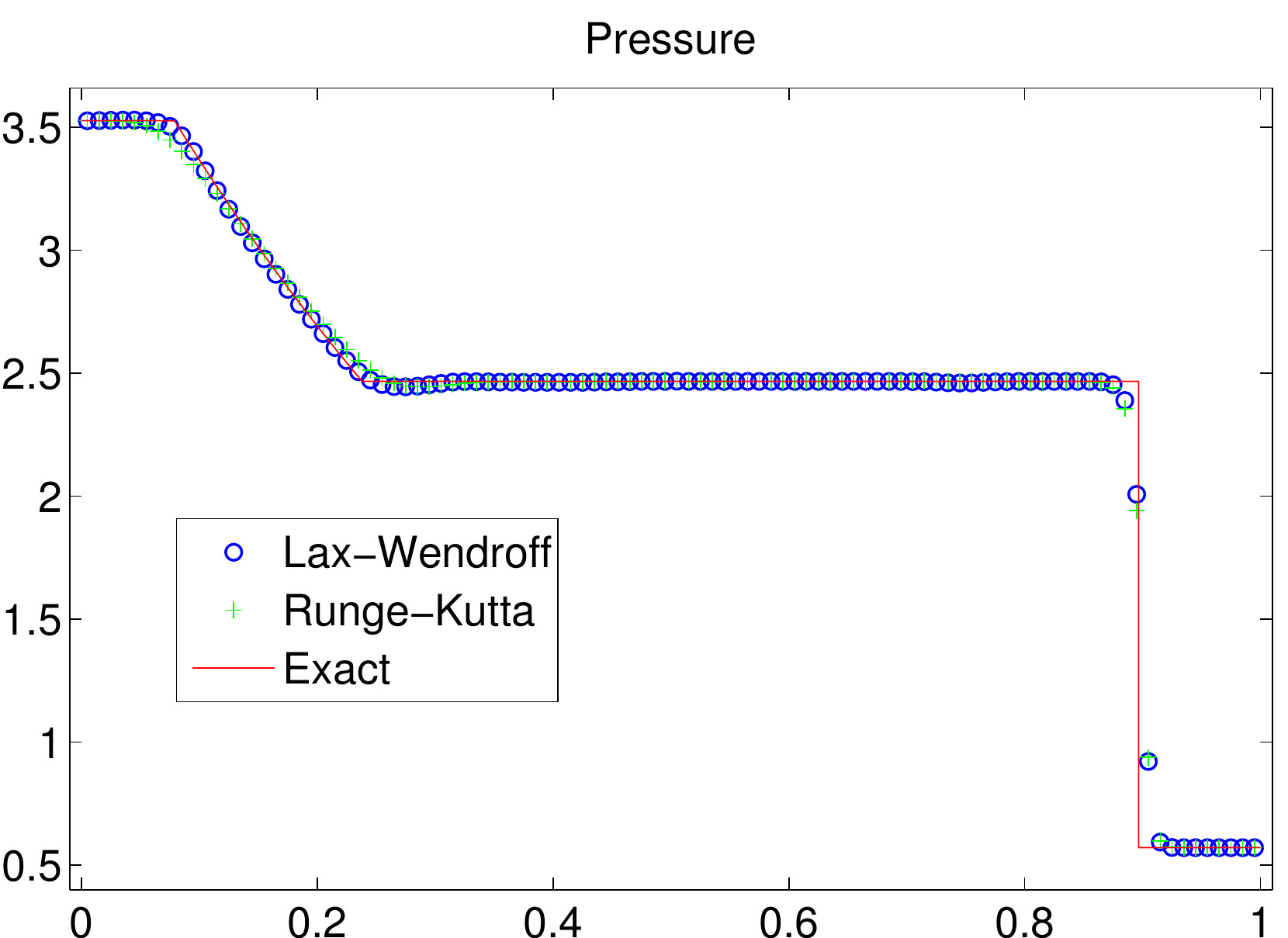}
\caption{Shock-tube Riemann problem.
Shown here are WENO simulations with Harten's initial conditions  \eqref{eqn:shock-harten}
for the shock tube problem.
We plot observable quantities from left to right:
density $\rho$, velocity $u$, and pressure $p$.
These results were obtained with a CFL number of $\nu = 0.4$, and 
$m_x = 100$ grid points.  
The exact solution (solid line)
is plotted underneath the results of our methods.
}
\label{fig:euler-shock-tube-harten-weno}
\end{figure}

\subsubsection{The 1D Euler equations: shock entropy}

Our final 1D test case is another problem that is popular in the 
literature \cite{ShuOsher89}.  The initial conditions are
\begin{align*}
   (\rho, u, p) =& \left( 3.857143, 2.629369, 10.3333 \right), \quad &x < -4, \\
   (\rho, u, p) =& \left( 1 + \epsilon \sin(5x), 0, 1 \right), \quad &x \geq -4,
\end{align*}
with a  computational domain of $[-5,5]$.  
The final time for this simulation is $t=1.8$.
With $\eps = 0$, this is a pure Mach 3 shock moving to the right.
We follow the common practice of setting $\eps = 0.2$.
Results for the new method are presented in Fig.~\ref{fig:euler-shock-entropy}.

\begin{figure}[!htb]
\centering
\includegraphics[width=0.3\textwidth]{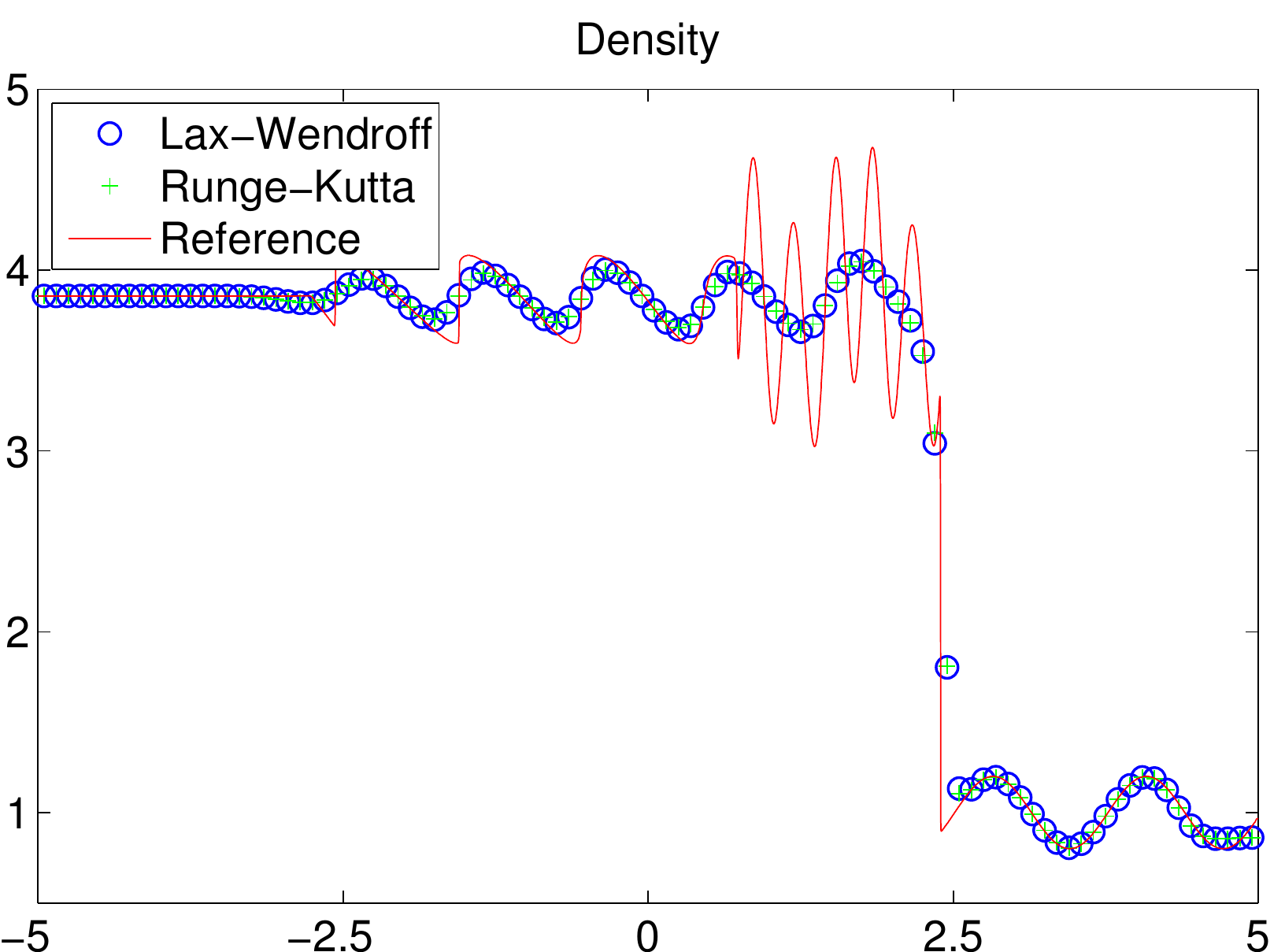}
\includegraphics[width=0.3\textwidth]{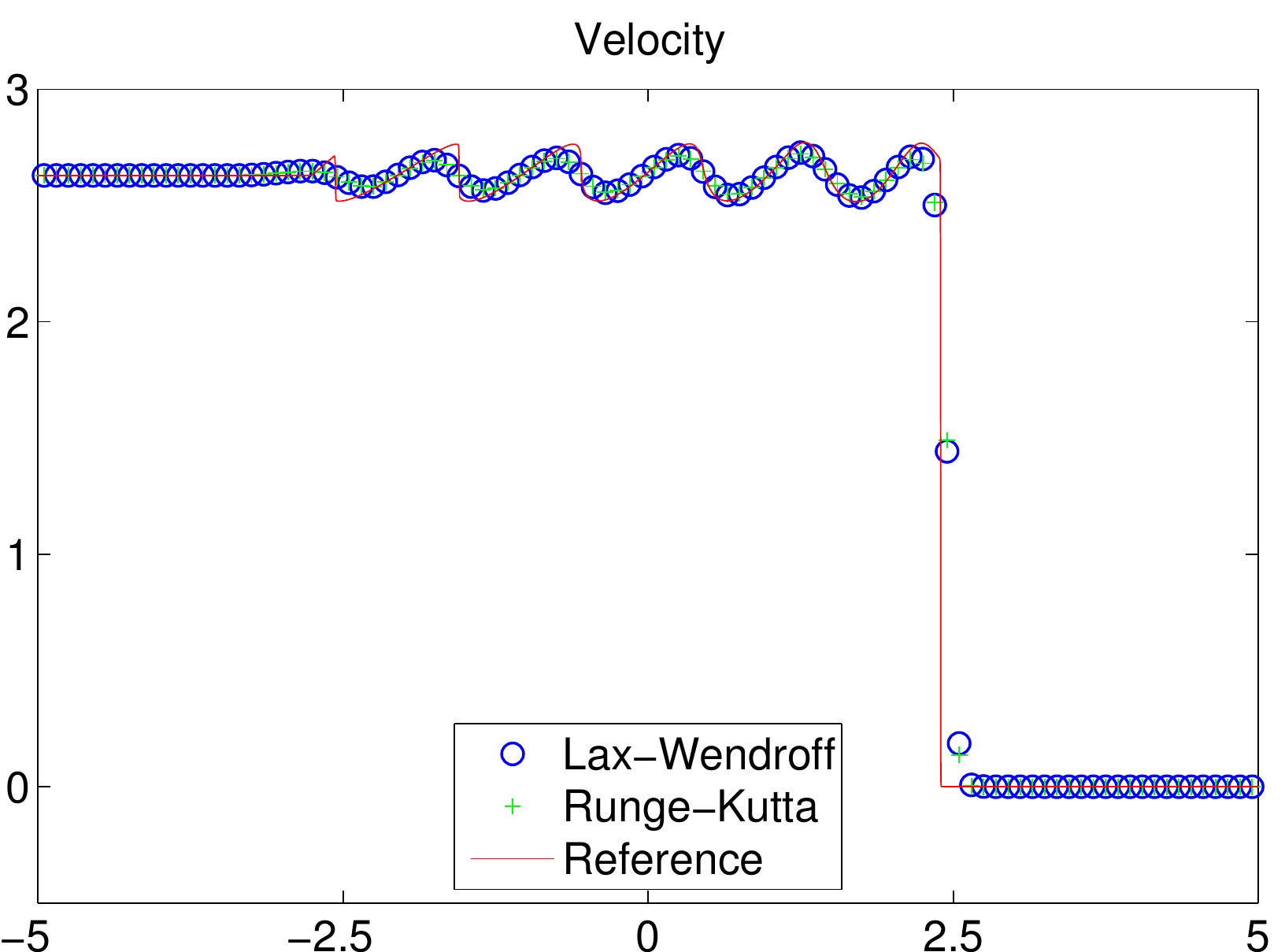}
\includegraphics[width=0.3\textwidth]{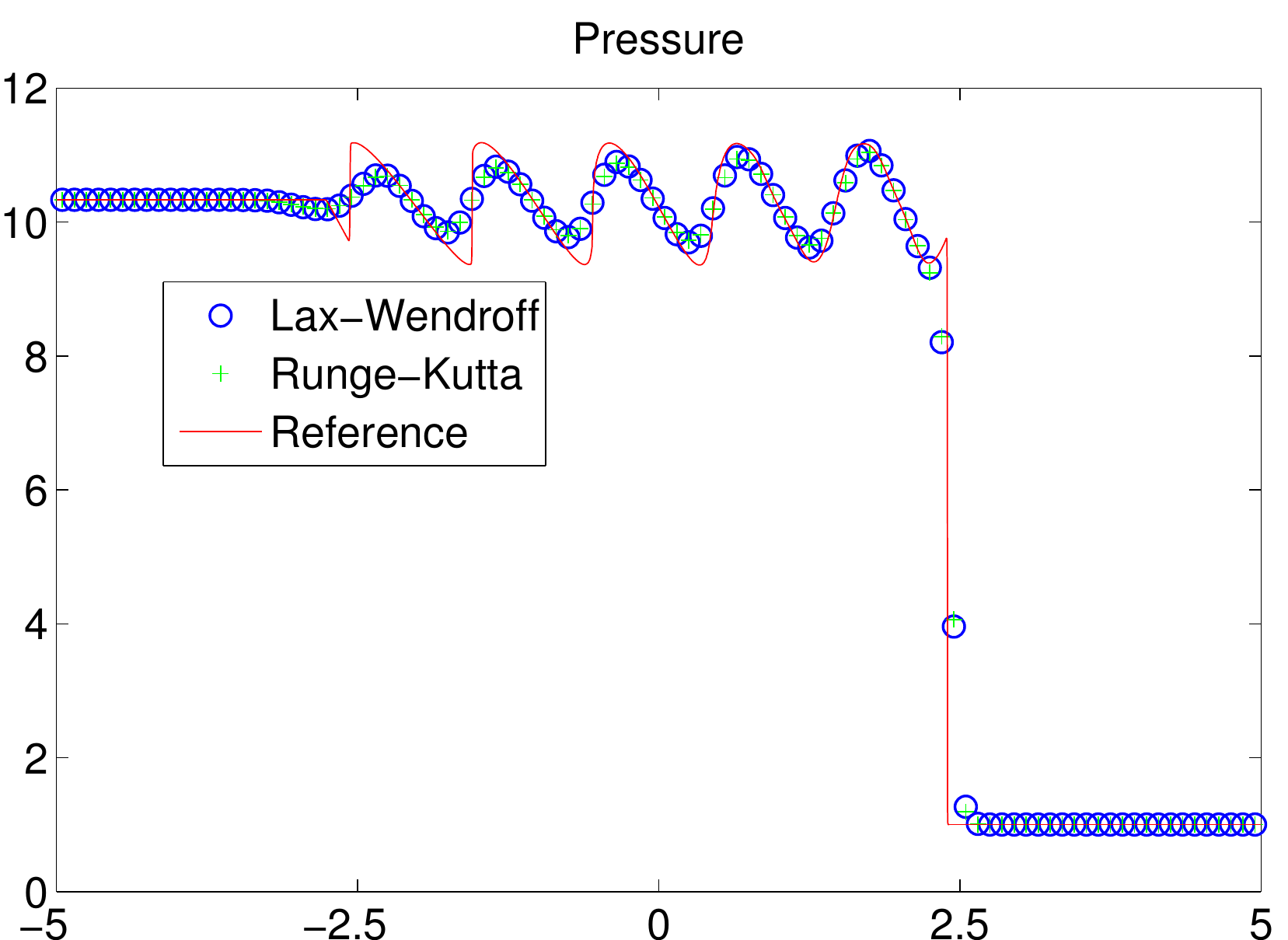}
\caption{1D Euler equations: shock entropy problem.
We plot observable quantities from left to right:
density $\rho$, velocity $u$, and pressure $p$.
These results were run with a CFL number of $\nu = 0.4$, and 
$m_x = 100$ grid points.  
Given that we do not have access to an exact solution, 
we plot a reference solution (solid line) constructed from a classical finite difference WENO 
simulation that uses the SSP-RK3 method described in Gottlieb and Shu \cite{GoShu98},
with $m_x = 6000$ points and a small CFL number of $\nu = 0.1$.  Observe that 
the Runge-Kutta method is indicating convergence to the correct entropy solution despite
the fact that no characteristic projections were used to compute each 
intermediate stage value.
\label{fig:euler-shock-entropy}
}
\end{figure}

\subsection{The 2D Euler equations}
In 2D the Euler equations become
\begin{equation}
\left( 
\begin{array}{c}
    \rho \\ \rho u \\ \rho v \\ \E
    \end{array}
    \right)_{, t}
    + 
    \left( 
    \begin{array}{c}
       \rho u \\ \rho u^2 + p \\ \rho u v \\ ( \E + p ) u
    \end{array}
    \right)_{, x}
    +
    \left( 
    \begin{array}{c}
       \rho v \\ \rho u v \\ \rho v^2 + p,  \\ ( \E + p ) v
    \end{array}
    \right)_{, y}
    = 0,
\end{equation}
where $u$ is the $x$-component of velocity, and $v$ is the $y$-component of
velocity.
Now, the total energy relies on both components of the velocity:
$\E = \frac{p}{\gamma-1} + \frac{1}{2} \rho \left(u^2 + v^2 \right)$.

\subsubsection{The 2D Euler equations: a smooth solution}
\label{subsubsec:2deuler-smooth}

We present convergence results for a smooth solution proposed elsewhere 
in the literature \cite{QiuShu03,QiuDumbserShu05}.  The initial conditions are
prescribed by
\begin{equation*}
    \left( \rho, u, v, p \right) = 
    \left(1 + 0.2 \sin(\pi(x+y)), 0.7, 0.3, 1.0 \right),
\end{equation*}
and we compute the solution on a periodic domain $\Omega = [0,2]\times[0,2]$.  
The exact solution has an evolving
density,
$\rho(t,x,y) = 1 + 0.2\sin\left(\pi(x+y - (u+v)t )\right)$, and constant
velocities $u = 0.7, v = 0.3$ and pressure $p = 1.0$.
A convergence study for this 2D problem is presented in 
Table \ref{table:euler-convergence}.  
In 2D, we define the relative $L^1$ error as
\begin{equation}
\label{eqn:error2d}
\text{Error } := 
   \frac{ { \dx \dy \sum_{i=1}^{m_x} \sum_{j=1}^{m_y} \left| q^n_{ij} - q(t^n, x_i, y_j ) \right| } }{ 
           \dx \dy \sum_{i=1}^{m_x} \sum_{j=1}^{m_y} |q(t^n, x_i, y_j)|  },
\end{equation}
and the CFL number is defined by
\begin{equation}
\label{eqn:cfl2d}
    \nu := \dt \max_{i,j} \left\{ \frac{ \alpha^*_{i-1/2,j} }{ \dx}, \frac{
    \alpha^*_{i,j-1/2} }{\dy} \right\}.
\end{equation}

\begin{table}
\centering
\normalsize
\caption{Euler equations: smooth solutions.
Shown here is a convergence study
for smooth solutions to 2D Euler equations as presented in 
\S\ref{subsubsec:2deuler-smooth}.  We use a uniform grid with 
$m_x = m_y$ grid points reported in the first column.  
Errors are defined as relative errors as in \eqref{eqn:error2d}.
We use a CFL number of $\nu = 0.4$ as defined in \eqref{eqn:cfl2d}.
The first two columns are the errors for the Taylor discretization of the PIF, and the last two
columns describe the errors for the RK4 discretization of the PIF.  
For comparison, in the middle two columns, we report data for
a classical finite difference WENO simulation run with the SSP-RK3 
method described in Gottlieb and Shu \cite{GoShu98}, and remark that the
errors for the third-order schemes are on par with each other.
The dominant error for the RK discretization
is a spatial error, and we run into machine round-off errors before being able to expose the formal fourth-order accuracy.
\label{table:euler-convergence} }
\footnotesize
\begin{tabular}{|r||c|c||c|c||c|c|}
\hline
\bf{Mesh} & \bf{PIF-Taylor} & \bf{Order} & \bf{SSP-RK3} & \bf{Order} & \bf{PIF-RK4} & \bf{Order} \\
\hline
\hline
$  50$ & $5.4101\times 10^{-06}$ & --- & $5.4514\times 10^{-06}$ & --- & $5.1974\times 10^{-06}$ & --- \\
\hline
$ 100$ & $1.9177\times 10^{-07}$ & $4.818$ & $1.9289\times 10^{-07}$ & $4.821$ & 
$1.6132 \times 10^{-07}$ & $5.010$ \\
\hline
$ 200$ & $8.8529\times 10^{-09}$ & $4.437$ & $8.8841\times 10^{-09}$ &
$4.440$ &
$4.9388 \times 10^{-09}$ & $5.030$ \\
\hline
$ 400$ & $6.3477\times 10^{-10}$ & $3.802$ & $6.3557\times 10^{-10}$ &
$3.805$ & $1.4352 \times 10^{-10}$ & $5.105$ \\
\hline
$ 800$ & $6.4601\times 10^{-11}$ & $3.297$ & $6.4636\times 10^{-11}$ &
$3.298$ &
$3.6152 \times 10^{-12}$ & $5.311$
\\
\hline
\end{tabular}

\end{table}

\subsubsection{The 2D Euler equations: double mach reflection}
\label{subsubsec:double-mach}

We now present results for the so-called double mach reflection originally proposed by
Woodward and Colella \cite{WoodwardColella84} that was intended to serve as a 
test problem to compare numerical methods.  
This problem has since become ubiquitous in
the literature \cite{JiangShu96,BaDiShu00,ReFlaShep03,ShiZhangYongShu03,TitarevToro05:JCP:systems}.
The initial
conditions describe a Mach-10 shock incident upon a single 
wedge (c.f. Fig.~$4$ in \cite{ReFlaShep03}).  The
computational domain is tilted, so that the 
wedge is positioned along the bottom of the grid.  The 
shock forms an oblique angle with the mesh, where it starts at the front of
the wedge located at
$(x,y) = (1/6,0)$, and continues up to the top of the computational domain
located at $y=1$.  
The initial conditions 
describe two constant values, one to the left and one to right of the shock as
\begin{equation}
\left(\rho, u, v, p \right) = 
\begin{cases}
   (1.4,\, 0,\, 0,\, 1.0)^T, \quad & \text{if } x < \frac{1}{6} + \frac{y}{\sqrt{3}}, \\
   (8.0,\, \frac{8.25\sqrt{3}}{2},\, -\frac{8.25}{2},\, 116.5)^T, \quad & \text{otherwise}.
\end{cases}
\end{equation}
Reflective boundary conditions are applied along the bottom edge when
$x > 1/6$, and the exact pre- and post-shock values are padded everywhere else.
In order to pad the correct boundary conditions along the top, we
require the exact location $x_s$ of the shock at time $t$ along the line $y=1$,
which is given by
$x_s(t) = 1/6+(20t+1)/\sqrt{3}$.
Results for this problem are presented in Fig.~\ref{fig:mach-reflection}.  
Given that
a consensus concerning what contour lines to plot has not been reached, we 
plot the 30 equally spaced contours from $\rho = 1.728$ to $\rho = 20.74$, as 
reported in one of the many simulations presented by the original authors \cite{WoodwardColella84}.

\begin{figure}[!htb]
\centering
\includegraphics[width=0.9\textwidth]{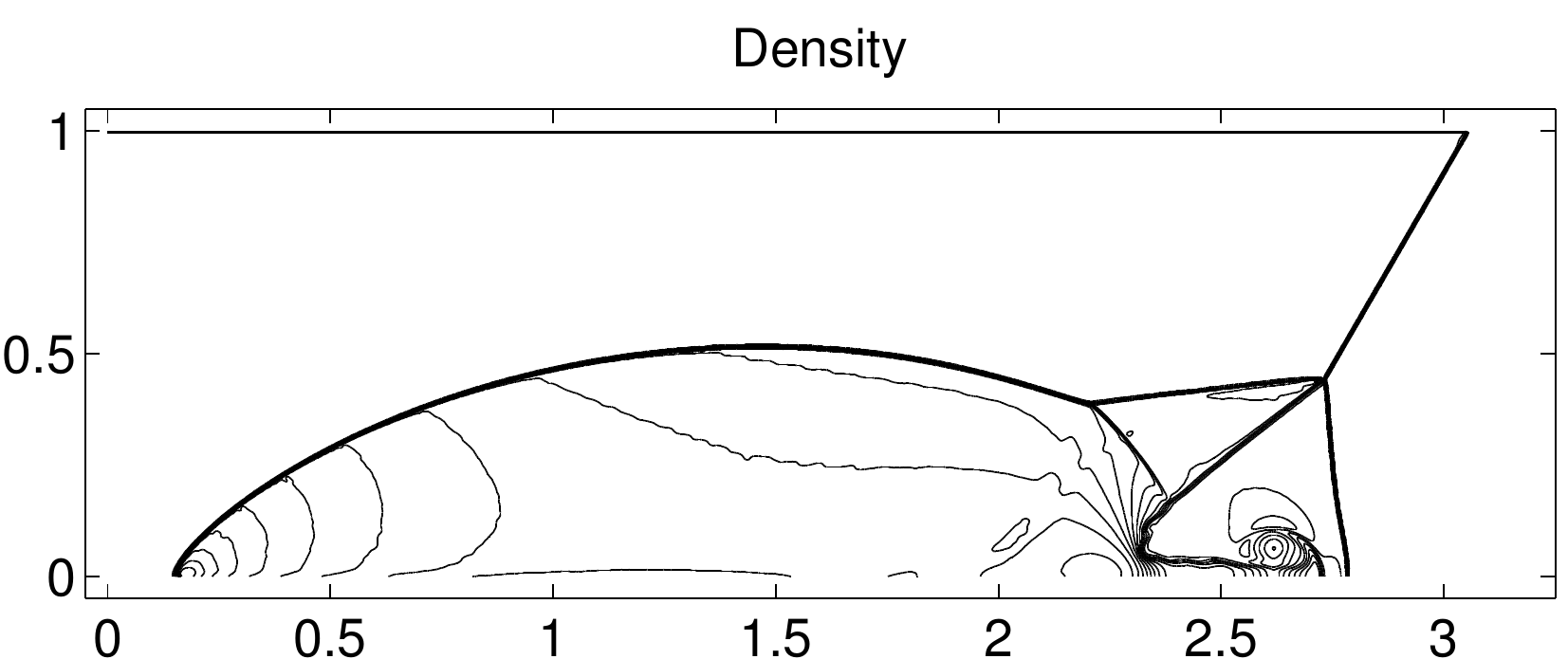}
\caption{Double mach reflection.  
Shown here are results for the Taylor method applied to the double mach
reflection problem described in \S\ref{subsubsec:double-mach}.  Results for the
RK discretization are nearly identical.  Minor deviations in the contour
lines are only visible under a high degree of magnification.
We run the simulation to a final time of $t=0.2$, and use a 
CFL number of $\nu = 0.4$.  A total of $(m_x,m_y) = (900,300)$ grid points are
used for this simulation.  
We plot 30 equally spaced contours from $\rho = 1.728$ to $\rho = 20.74$.
The results are in agreement with 
simulations produced by explicit SSP time integrators.
}
\label{fig:mach-reflection}
\end{figure}

\section{Conclusions}
\label{sec:conclusions}

We have formulated and presented results for the 
\emph{Picard integral formulation} of the finite difference WENO method.  The new formulation allows
us to step back from the classical formulations and include the spatial
discretization as part of the temporal discretizations. 
We have demonstrated how Taylor and Runge-Kutta methods can be
developed from this new vantage, and have introduced results for the proposed
formulation in one- and two-dimensions, that indicate the new methods compete
with current state of the art technology.  Future work will include
investigating positivity-preserving limiters for finite difference methods, 
as well as other temporal discretizations of the time-averaged fluxes 
including multiderivative and alternative flux modifications.


{\bf Acknowledgments} We would like to thank the anonymous reviewers for their
thoughtful comments and encouragement to include more extensive results.
In addition, we would like to thank Gwendolyn Miller Seal for
creating one of the figures used in this manuscript, and Nathan Collins for
aid in manipulating the symbolic expressions used to generate source code
for the simulations.  Finally, we would like to thank Scott Moe for
providing feedback on the revised manuscript.

\begin{appendix}

\section{WENO reconstruction}
\label{app:weno-reconstruct}

For completeness, we present the coefficients required to
reproduce the results in this paper.  
We restrict our attention to the fifth-order case, which
uses a five point stencil shifted to the left or 
right of each cell interface:
\begin{align*}
   u_{i+1/2}^+ &:= WENO5^+[ \u_{i-2}, \u_{i-1}, \u_{i},   \u_{i+1}, \u_{i+2} ], \\
   u_{i+1/2}^- &:= WENO5^-[ \u_{i-1}, \u_{i},   \u_{i+1}, \u_{i+2}, \u_{i+3} ].
\end{align*}
Here, the known values are the cell averages 
\begin{equation*}
    \u_j = \frac{1}{\dx} \int_{x_j-\dx/2}^{x_j+\dx/2} u(x)\, dx
\end{equation*}
for some scalar function $u$.
The purpose of having a ``$+$'' and ``$-$'' value is to define upwind stencils
that are numerically stable \cite{book:Le02}.
We define coefficients for the function $WENO5^+$, and
by symmetry, we define
\begin{align}
   WENO5^-[ \u_{i-1}, \u_{i},   \ldots \u_{i+3} ] :=
   WENO5^+[ \u_{i+3}, \u_{i+2}, \ldots \u_{i-1} ].
\end{align}
%
Three sub-stencils define quadratic polynomials that offer competing third-order
accurate values for $u(x_{i+1/2})$:
\begin{subequations}
\begin{align}
\label{eqn:weno-interp0}
   u^{(0)}_{i+1/2} &= \phantom{-} 
   \frac{1}{3} \u_{i-2} -  \frac{7}{6} \u_{i-1} + \frac{11}{6} \u_i, \\
\label{eqn:weno-interp1}
   u^{(1)}_{i+1/2} &= - 
   \frac{1}{6} \u_{i-1} +  \frac{5}{6} \u_{i\phantom{-1}} + \frac1 3 \u_{i+1}, \\
\label{eqn:weno-interp2}
   u^{(2)}_{i+1/2} &= \phantom{-} 
   \frac{1}{3} \u_{i\phantom{-1}} +  \frac{5}{6} \u_{i+1} - \frac{1}{6} \u_{i+2}.
\end{align}
\end{subequations}

%
A linear combination of \eqref{eqn:weno-interp0}-\eqref{eqn:weno-interp2} yields
a fifth-order accurate approximation
$u(x_{i+1/2}) \approx   \gamma_0 u^{(0)}_{i+1/2} +
                        \gamma_1 u^{(1)}_{i+1/2} +
                        \gamma_2 u^{(2)}_{i+1/2}$
with \emph{linear weights} 
$\gamma_j \in \left\{1/10,\, 3/5,\, 3/10\right\}$.
%
The WENO procedure replaces the linear weights $\gamma_j$ with nonlinear 
weights $\omega_j$ based on 
the \emph{smoothness indicators} $\beta_j$.
%
In fifth-order WENO, the indicators are
\begin{align}
\begin{aligned}
   \beta_0 &=  \frac{13}{12} \left(  \u_{i-2}- 2\u_{i-1} +  \u_i     \right)^2
             + \frac{1}{4}   \left(  \u_{i-2}- 4\u_{i-1} + 3\u_i     \right)^2, \\
   \beta_1 &=  \frac{13}{12} \left(  \u_{i-1}- 2\u_{i}   +  \u_{i+1} \right)^2
             + \frac{1}{4}   \left(  \u_{i-1}-  \u_{i+1}             \right)^2, \\
   \beta_2 &=  \frac{13}{12} \left(  \u_{i}  - 2\u_{i+1} +  \u_{i+2} \right)^2
             + \frac{1}{4}   \left( 3\u_{i}  - 4\u_{i+1} +  \u_{i+2} \right)^2.
\end{aligned}
\end{align}
%
The Jiang and Shu weights \cite{JiangShu96} are defined by
\begin{align}
   \omega_k = \frac{\tilde{\omega}_k}{ \sum_{l=0}^2 \tilde{\omega}_l }, 
       \quad
   \tilde{\omega}_k = \frac{\gamma_k}{\left( \beta_k + \eps \right)^p }.
\end{align}
We use the power parameter $p=2$ and regularization parameter
$\eps = 10^{-12}$ for all of our simulations.
With these definitions in place, the final reconstructed value is defined as
\begin{align}
 WENO5^+[ \u_{i-2}, \ldots \u_{i+2} ] := 
 \omega_0\, u^{(0)}_{i+1/2} +
 \omega_1\, u^{(1)}_{i+1/2} +
 \omega_2\, u^{(2)}_{i+1/2}.
\end{align}


\end{appendix}


\bibliographystyle{ieeetr}      

\bibliography{Edited-References}

\begin{thebibliography}{10}

\bibitem{HaWa73}
E.~Hairer and G.~Wanner, ``Multistep-multistage-multiderivative methods of
  ordinary differential equations,'' {\em Computing (Arch. Elektron. Rechnen)},
  vol.~11, no.~3, pp.~287--303, 1973.

\bibitem{LxW60}
P.~Lax and B.~Wendroff, ``Systems of conservation laws,'' {\em Comm. Pure Appl.
  Math.}, vol.~13, pp.~217--237, 1960.

\bibitem{HaEnOsCh87}
A.~Harten, B.~Engquist, S.~Osher, and S.~R. Chakravarthy, ``Uniformly
  high-order accurate essentially nonoscillatory schemes. {III},'' {\em J.
  Comput. Phys.}, vol.~71, no.~2, pp.~231--303, 1987.

\bibitem{TitarevToro02-ader}
E.~F. Toro and V.~A. Titarev, ``{ADER}: Arbitrary high order {G}odunov
  approach,'' {\em Journal of Scientific Computing}, vol.~17, no.~1-4,
  pp.~609--618, 2002.

\bibitem{ToroTitarev02}
E.~F. Toro and V.~A. Titarev, ``Solution of the generalized {R}iemann problem
  for advection-reaction equations,'' {\em R. Soc. Lond. Proc. Ser. A Math.
  Phys. Eng. Sci.}, vol.~458, no.~2018, pp.~271--281, 2002.

\bibitem{ToroTitarev05:JSP}
E.~F. Toro and V.~A. Titarev, ``T{VD} fluxes for the high-order {ADER}
  schemes,'' {\em J. Sci. Comput.}, vol.~24, no.~3, pp.~285--309, 2005.

\bibitem{TitarevToro05:JCP:systems}
V.~A. Titarev and E.~F. Toro, ``A{DER} schemes for three-dimensional non-linear
  hyperbolic systems,'' {\em J. Comput. Phys.}, vol.~204, no.~2, pp.~715--736,
  2005.

\bibitem{DuKa07}
M.~Dumbser and M.~K{\"a}ser, ``Arbitrary high order non-oscillatory finite
  volume schemes on unstructured meshes for linear hyperbolic systems,'' {\em
  J. Comput. Phys.}, vol.~221, no.~2, pp.~693--723, 2007.

\bibitem{ToTi06}
E.~F. Toro and V.~A. Titarev, ``Derivative {R}iemann solvers for systems of
  conservation laws and {ADER} methods,'' {\em J. Comput. Phys.}, vol.~212,
  no.~1, pp.~150--165, 2006.

\bibitem{NoFi12}
M.~R. Norman and H.~Finkel, ``Multi-moment {ADER}-{T}aylor methods for systems
  of conservation laws with source terms in one dimension,'' {\em J. Comput.
  Phys.}, vol.~231, no.~20, pp.~6622--6642, 2012.

\bibitem{BaDiMeDuHuXu13}
D.~S. Balsara, C.~Meyer, M.~Dumbser, H.~Du, and Z.~Xu, ``Efficient
  implementation of {ADER} schemes for {E}uler and magnetohydrodynamical flows
  on structured meshes---speed comparisons with {R}unge-{K}utta methods,'' {\em
  J. Comput. Phys.}, vol.~235, pp.~934--969, 2013.

\bibitem{QiuShu03}
J.~Qiu and C.-W. Shu, ``Finite difference {WENO} schemes with
  {L}ax-{W}endroff-type time discretizations,'' {\em SIAM J. Sci. Comput.},
  vol.~24, no.~6, pp.~2185--2198, 2003.

\bibitem{LiuChengShu09}
W.~Liu, J.~Cheng, and C.-W. Shu, ``High order conservative {L}agrangian schemes
  with {L}ax-{W}endroff type time discretization for the compressible {E}uler
  equations,'' {\em J. Comput. Phys.}, vol.~228, no.~23, pp.~8872--8891, 2009.

\bibitem{LuQiu11}
C.~Lu and J.~Qiu, ``Simulations of shallow water equations with finite
  difference {L}ax-{W}endroff weighted essentially non-oscillatory schemes,''
  {\em J. Sci. Comput.}, vol.~47, no.~3, pp.~281--302, 2011.

\bibitem{JiangShuZhang13}
Y.~Jiang, C.-W. Shu, and M.~Zhang, ``An alternative formulation of finite
  difference weighted {ENO} schemes with {L}ax-{W}endroff time discretization
  for conservation laws,'' {\em SIAM J. Sci. Comput.}, vol.~35, no.~2,
  pp.~A1137--A1160, 2013.

\bibitem{QiuDumbserShu05}
J.~Qiu, M.~Dumbser, and C.-W. Shu, ``The discontinuous {G}alerkin method with
  {L}ax-{W}endroff type time discretizations,'' {\em Comput. Methods Appl.
  Mech. Eng.}, vol.~194, no.~42-44, pp.~4528--4543, 2005.

\bibitem{DuMu06}
M.~Dumbser and C.-D. Munz, ``Building blocks for arbitrary high order
  discontinuous {G}alerkin schemes,'' {\em J. Sci. Comput.}, vol.~27, no.~1-3,
  pp.~215--230, 2006.

\bibitem{TaDuBaDiMu07}
A.~Taube, M.~Dumbser, D.~S. Balsara, and C.-D. Munz, ``Arbitrary high-order
  discontinuous {G}alerkin schemes for the magnetohydrodynamic equations,''
  {\em J. Sci. Comput.}, vol.~30, no.~3, pp.~441--464, 2007.

\bibitem{Qiu07:numcomp}
J.~Qiu, ``A numerical comparison of the {L}ax-{W}endroff discontinuous
  {G}alerkin method based on different numerical fluxes,'' {\em J. Sci.
  Comput.}, vol.~30, no.~3, pp.~345--367, 2007.

\bibitem{GaDuHiMuCl11}
G.~Gassner, M.~Dumbser, F.~Hindenlang, and C.-D. Munz, ``Explicit one-step time
  discretizations for discontinuous {G}alerkin and finite volume schemes based
  on local predictors,'' {\em J. Comput. Phys.}, vol.~230, no.~11,
  pp.~4232--4247, 2011.

\bibitem{DuBaDiToMuDi08}
M.~Dumbser, D.~S. Balsara, E.~F. Toro, and C.-D. Munz, ``A unified framework
  for the construction of one-step finite volume and discontinuous {G}alerkin
  schemes on unstructured meshes,'' {\em J. Comput. Phys.}, vol.~227, no.~18,
  pp.~8209--8253, 2008.

\bibitem{VanLeer79}
B.~van Leer, ``Towards the ultimate conservative difference scheme. {V}. {A}
  second-order sequel to {G}odunov's method,'' {\em J. Comput. Phys.}, vol.~32,
  no.~1, pp.~101--136, 1979.

\bibitem{Ha84}
A.~Harten, ``On a class of high resolution total-variation-stable
  finite-difference schemes,'' {\em SIAM J. Numer. Anal.}, vol.~21, no.~1,
  pp.~1--23, 1984.

\bibitem{Ro87}
A.~Rodionov, ``Methods of increasing the accuracy in {G}odunov's scheme,'' {\em
  {USSR} Computational Mathematics and Mathematical Physics}, vol.~27, no.~6,
  pp.~164 -- 169, 1987.

\bibitem{BeCoTr89}
J.~B. Bell, P.~Colella, and J.~A. Trangenstein, ``Higher order {G}odunov
  methods for general systems of hyperbolic conservation laws,'' {\em J.
  Comput. Phys.}, vol.~82, no.~2, pp.~362--397, 1989.

\bibitem{Me90}
I.~Men'shov, ``Increasing the order of approximation of {G}odunov's scheme
  using solutions of the generalized {R}iemann problem,'' {\em {USSR}
  Computational Mathematics and Mathematical Physics}, vol.~30, no.~5, pp.~54
  -- 65, 1990.

\bibitem{QiuShu2011}
J.-M. Qiu and C.-W. Shu, ``Conservative high order semi-{L}agrangian finite
  difference {WENO} methods for advection in incompressible flow,'' {\em J.
  Comput. Phys.}, vol.~230, no.~4, pp.~863--889, 2011.

\bibitem{ShuOsher87}
C.-W. Shu and S.~Osher, ``Efficient implementation of essentially
  nonoscillatory shock-capturing schemes,'' {\em J. Comput. Phys.}, vol.~77,
  no.~2, pp.~439--471, 1988.

\bibitem{ShuOsher89}
C.-W. Shu and S.~Osher, ``Efficient implementation of essentially
  nonoscillatory shock-capturing schemes. {II},'' {\em J. Comput. Phys.},
  vol.~83, no.~1, pp.~32--78, 1989.

\bibitem{LiuOsherChan94}
X.-D. Liu, S.~Osher, and T.~Chan, ``Weighted essentially non-oscillatory
  schemes,'' {\em J. Comput. Phys.}, vol.~115, no.~1, pp.~200--212, 1994.

\bibitem{JiangShu96}
G.-S. Jiang and C.-W. Shu, ``Efficient implementation of weighted {ENO}
  schemes,'' {\em J. Comput. Phys.}, vol.~126, no.~1, pp.~202--228, 1996.

\bibitem{Shu97}
C.-W. Shu, ``Essentially non-oscillatory and weighted essentially
  non-oscillatory schemes for hyperbolic conservation laws,'' in {\em Advanced
  numerical approximation of nonlinear hyperbolic equations ({C}etraro, 1997)},
  vol.~1697 of {\em Lecture Notes in Math.}, pp.~325--432, Berlin: Springer,
  1998.

\bibitem{Shu09}
C.-W. Shu, ``{High order weighted essentially nonoscillatory schemes for
  convection dominated problems},'' {\em SIAM Rev.}, vol.~51, no.~1,
  pp.~82--126, 2009.

\bibitem{Roe81}
P.~L. Roe, ``Approximate {R}iemann solvers, parameter vectors, and difference
  schemes,'' {\em J. Comput. Phys.}, vol.~43, no.~2, pp.~357--372, 1981.

\bibitem{BaDiShu00}
D.~S. Balsara and C.-W. Shu, ``Monotonicity preserving weighted essentially
  non-oscillatory schemes with increasingly high order of accuracy,'' {\em J.
  Comput. Phys.}, vol.~160, no.~2, pp.~405--452, 2000.

\bibitem{VuSo02}
S.~Vukovic and L.~Sopta, ``E{NO} and {WENO} schemes with the exact conservation
  property for one-dimensional shallow water equations,'' {\em J. Comput.
  Phys.}, vol.~179, no.~2, pp.~593--621, 2002.

\bibitem{SeGuCh13}
D.~C. Seal, Y.~G\"{u}\c{c}l\"{u}, and A.~J. Christlieb, ``High-order
  multiderivative time integrators for hyperbolic conservation laws,'' {\em J.
  Sci. Comput.}, pp.~1--40, 2013.

\bibitem{WaSp07}
R.~Wang and R.~J. Spiteri, ``Linear instability of the fifth-order {WENO}
  method,'' {\em SIAM J. Numer. Anal.}, vol.~45, no.~5, pp.~1871--1901
  (electronic), 2007.

\bibitem{MoMaRu11}
M.~Motamed, C.~B. Macdonald, and S.~J. Ruuth, ``On the linear stability of the
  fifth-order {WENO} discretization,'' {\em J. Sci. Comput.}, vol.~47, no.~2,
  pp.~127--149, 2011.

\bibitem{book:Le02}
R.~LeVeque, {\em {F}inite {V}olume {M}ethods for {H}yperbolic {P}roblems}.
\newblock Cambridge University Press, 2002.

\bibitem{Williamson80}
J.~H. Williamson, ``Low-storage {R}unge-{K}utta schemes,'' {\em J. Comput.
  Phys.}, vol.~35, no.~1, pp.~48--56, 1980.

\bibitem{Ke08}
D.~I. Ketcheson, ``Highly efficient strong stability-preserving {R}unge-{K}utta
  methods with low-storage implementations,'' {\em SIAM J. Sci. Comput.},
  vol.~30, no.~4, pp.~2113--2136, 2008.

\bibitem{Ke10}
D.~I. Ketcheson, ``Runge-{K}utta methods with minimum storage
  implementations,'' {\em J. Comput. Phys.}, vol.~229, no.~5, pp.~1763--1773,
  2010.

\bibitem{NieDieBu12}
J.~Niegemann, R.~Diehl, and K.~Busch, ``Efficient low-storage {R}unge-{K}utta
  schemes with optimized stability regions,'' {\em J. Comput. Phys.}, vol.~231,
  no.~2, pp.~364--372, 2012.

\bibitem{GoShu98}
S.~Gottlieb and C.-W. Shu, ``Total variation diminishing {R}unge-{K}utta
  schemes,'' {\em Math. Comp.}, vol.~67, no.~221, pp.~73--85, 1998.

\bibitem{Harten78}
A.~Harten, ``The artificial compression method for computation of shocks and
  contact discontinuities. {III}. {S}elf-adjusting hybrid schemes,'' {\em Math.
  Comp.}, vol.~32, no.~142, pp.~363--389, 1978.

\bibitem{WoodwardColella84}
P.~Woodward and P.~Colella, ``The numerical simulation of two-dimensional fluid
  flow with strong shocks,'' {\em J. Comput. Phys.}, vol.~54, no.~1,
  pp.~115--173, 1984.

\bibitem{ReFlaShep03}
J.-F. Remacle, J.~E. Flaherty, and M.~S. Shephard, ``An adaptive discontinuous
  {G}alerkin technique with an orthogonal basis applied to compressible flow
  problems,'' {\em SIAM Rev.}, vol.~45, no.~1, pp.~53--72 (electronic), 2003.

\bibitem{ShiZhangYongShu03}
J.~Shi, Y.-T. Zhang, and C.-W. Shu, ``Resolution of high order {WENO} schemes
  for complicated flow structures,'' {\em J. Comput. Phys.}, vol.~186, no.~2,
  pp.~690--696, 2003.

\end{thebibliography}

\end{document}